\documentclass[11pt]{article}
\usepackage[utf8]{inputenc}
\usepackage{amsmath,amssymb,amsthm,amsfonts}
\usepackage{hyperref}
\usepackage{graphicx}
\usepackage{enumitem}
\usepackage{xspace}
\usepackage{xcolor}
\usepackage{textcomp}
\usepackage[title]{appendix}
\usepackage{booktabs}
\newtheorem{theorem}{Theorem}
\newtheorem{corollary}{Corollary}
\newtheorem{lemma}{Lemma}
\newtheorem{proposition}{Proposition}
\theoremstyle{definition}
\newtheorem{example}{Example}
\newtheorem{remark}{Remark}
\newtheorem{definition}{Definition}
\newtheorem{assumption}{Assumption}

\newcommand{\VU}{\(\mathcal{VU}\)\xspace} 

\newcommand{\func}[1]{\mathrm {#1}\,}
\DeclareUnicodeCharacter{00A0}{ }
\DeclareMathOperator*{\argmin}{arg\,min}

\DeclareMathOperator{\dom}{dom}

\DeclareMathOperator{\ri}{ri}
\DeclareMathOperator{\inte}{int}
\DeclareMathOperator{\Ker}{Ker}
\DeclareMathOperator{\ext}{ext}
\DeclareMathOperator{\aff}{aff}
\DeclareMathOperator{\co}{co}
\DeclareMathOperator{\spn}{span}

\DeclareMathOperator{\pos}{pos}
\mathchardef\hyphen="2D 

\newcommand{\overbar}[1]{\mkern 1.5mu\overline{#1\mkern-1.5mu}\mkern 1.5mu}

\newcommand{\mcb}[1]{\left\{#1\right\}}
\newcommand{\fa}{\forall \,}

\newcommand{\mang}[1]{\left\langle  #1\right\rangle}
\newcommand{\mbk}[1]{\left(#1\right)}

\newcommand{\urightarrow}[2][]{%
  \mathrel{%
    \xrightarrow[{\raisebox{0.5pt}[0pt][0pt]{$\scriptstyle#1$}}]{#2}%
  }%
}

\newcommand\RR{\mathbb{R}} 
\newcommand{\U}{\(\mathcal{U}\)\xspace}

\newcommand\infconv{\mathbin{
  \setbox0=\hbox{$\scriptstyle\vee$}\raise-0.55\ht0\hbox{$\scriptstyle\vee$}
  \setbox0=\hbox{$\scriptstyle\vee$}\hskip-0.5\wd0
  \setbox0=\hbox{$\scriptscriptstyle +$}\hskip-0.5\wd0
  \setbox0=\hbox{$\scriptstyle\vee$}\raise0.55\ht0\hbox{$\scriptscriptstyle +$}}}
\newcommand{\inv}{^{\raisebox{.2ex}{$\scriptscriptstyle-1$}}}

\providecommand{\norm}[1]{\lVert#1\rVert_{\infty}}
\providecommand{\normt}[1]{\lVert#1\rVert}
\providecommand{\normo}[1]{\lVert#1\rVert_{1}}
\providecommand{\abs}[1]{\lvert#1\rvert}

\newcommand\restr[2]{{
  \left.\kern-\nulldelimiterspace 
  #1 
  \vphantom{\big|} 
  \right|_{#2} 
  }}

\newcommand{\lowoverline}[1]{%
  \overline{\smash{#1}\vphantom{i  }}\vphantom{#1}%
}

\makeatletter
\AtBeginDocument{%
  \raggedbottom
  \def\@textbottom{\vskip \z@ \@plus 1pt}%
  \let\@texttop\relax
  \vbadness=10001
}
\makeatother
\begin{document}

\title{$\varepsilon$-$\mathcal{VU}$ spaces and approximate leaves\\ for sublinear composite functions}
\author{Shuai Liu\\
School of Mathematics and Statistics,\\
Nanfang College, Guangzhou\\
No.\ 882, Wenquan Avenue,\\
Conghua District, Guangzhou 510970, China\\
\texttt{shuai0liu@gmail.com}}
\date{}
\maketitle

\begin{abstract}
Let \(h\) be proper, closed and sublinear, and let \(F\) be of class \(C^1\). The exact \(\mathcal{VU}\) calculus of the composite \(f=h\circ F\) is taken as known. This paper studies enlargements \(\bar\partial_\varepsilon h\subset\partial_\varepsilon h\) and the resulting \(\varepsilon\)-spaces of \(f\). Nested inclusions give one-sided Painlev\'e--Kuratowski limits of \(\mathcal{V}_\varepsilon\) and \(\mathcal{U}_\varepsilon\) at points of the lineality space \(\mathcal{U}h(0)\). An additional transversality transfers those limits through \(F\); without it only the exact chain rule is available. Approximate leaves \(\mathcal{M}_\varepsilon(x)=F^{-1}\bigl(F(x)+\mathcal{U}_\varepsilon h(F(x))\bigr)\) are inner semicontinuous under the nesting \(\mathcal{U}h(F(\bar x))\subset\mathcal{U}_\varepsilon h(F(x))\). For a polyhedral outer function the activity set freezes along the exact leaf for small \(\varepsilon\), so the dimension identity holds on that leaf and not necessarily in a full neighbourhood of \((\bar x,0)\). Partial smoothness of \(h\) relative to \(\mathcal{U}h(0)\) extends a known result from the origin to the lineality space; it is not used by the transfer theorems.
\end{abstract}

\noindent\textbf{Keywords.} $\mathcal{VU}$-decomposition, partial smoothness, sublinear function, polyhedral function, composite function, subdifferential enlargement.

\noindent\textbf{MSC.} 49J52, 49J53, 58C06, 49K40.

\section{Introduction}
A composite \(f=h\circ F\) with \(h\) proper, closed and sublinear and \(F\) of class \(C^1\) is nondifferentiable, but the nondifferentiability is confined to a subspace. For such an \(h\) the subdifferential \(\partial h(x)\) is a face of \(\partial h(0)\), so the subspaces are nested: \(\mathcal{V}h(x)\subset\mathcal{V}h(0)\) and \(\mathcal{U}h(0)\subset\mathcal{U}h(x)\) whenever \(\partial h(x)\neq\emptyset\). The lineality space is \(\mathcal{U}h(0)\).

The \(\mathcal{VU}\)-decomposition is that of~\cite{Lemarechal2000}; partial smoothness relative to an active manifold is that of~\cite{Lewis2002active}; the identification of the normal space with \(\mathcal{V}\) is normal sharpness. For a sublinear function,~\cite{hare2006functions} shows partial smoothness at the origin relative to \(\mathcal{U}h(0)\). The nondegeneracy condition for a composite \(g\circ F\) is the same transversality. The exact \(\mathcal{VU}\) identities for \(f=h\circ F\), including the set \(\mathcal{R}_{h,F}\), are those of~\cite{Liu2026vu} and are not reproved. The extreme-point enlargement and the polyhedral continuity statements are those of~\cite{Liu2019Subdifferential}. The bivariate \(\mathcal{U}\)-Lagrangian and the proximal-gradient \(\mathcal{VU}\)-method are~\cite{Liu2025}; the prox-regular extension of the \(\mathcal{VU}\)-decomposition is~\cite{Liu2020}.

What is added is the passage from an enlargement \(\bar\partial_\varepsilon h\subset\partial_\varepsilon h\) to the composite \(\varepsilon\)-spaces and to approximate leaves, under the restriction that the outer function be sublinear. The restriction is the source of the nesting: for \(h\in\mathcal{H}\) the subdifferential at \(x\) is a face of \(\partial h(0)\), so the subspaces are nested and the one-sided limits exist at every point of the lineality space. The composites that fall in this class include the group lasso, an unbiased rectifier layer, and the nuclear norm; a biased layer does not, since translation destroys positive homogeneity. The \(\varepsilon\)-step beyond~\cite{Liu2019Subdifferential} is the passage of those limits through \(F\), and the approximate leaf, neither of which is in that account. Precisely:
\begin{itemize}
\item
nested inclusions give one-sided Painlev\'e--Kuratowski limits of \(\mathcal{V}_\varepsilon h\) and \(\mathcal{U}_\varepsilon h\) at points of the lineality space \(\mathcal{U}h(0)\); they are not Hausdorff continuity of the Grassmannian of \(\mathcal{U}_\varepsilon\) off the leaf;
\item
\textbf{an additional transversality, beyond membership in \(\mathcal{R}_{h,F}\), is shared by Theorems~\ref{thm:eps-transfer} and~\ref{thm:Meps}; it transfers the limits through \(F\) (Theorem~\ref{thm:eps-transfer}) and is the surjectivity hypothesis of the implicit-function step in Theorem~\ref{thm:Meps}. It is the nonseparability hypothesis of \cite[Theorem~4.32(a)]{Rockafellar1998} and the horizon qualification of \cite[Theorem~5.53(b)]{Rockafellar1998}, not a new condition. Without it only the exact chain rule is available;}
\item
\textbf{approximate leaves \(\mathcal{M}_\varepsilon(x)=F^{-1}\bigl(F(x)+\mathcal{U}_\varepsilon h(F(x))\bigr)\) are inner semicontinuous under the nesting \(\mathcal{U}h(F(\bar x))\subset\mathcal{U}_\varepsilon h(F(x))\), and converge in the Painlev\'e--Kuratowski sense once the dimension is frozen (Theorem~\ref{thm:Meps});}
\item
for a polyhedral outer function the activity set freezes along the exact leaf for small \(\varepsilon\) (Proposition~\ref{prop:Meps-poly}), so the dimension identity holds on that leaf and not necessarily in a full neighbourhood of \((\bar x,0)\);
\item
partial smoothness of \(h\in\mathcal{H}\) relative to \(\mathcal{U}h(0)\) extends~\cite{hare2006functions} from the origin to the lineality space. No algorithm is proposed.
\end{itemize}

Section~2 fixes the notation, the class \(\mathcal{H}\), and the enlargement. Section~3 records the properties of \(\mathcal{H}\), including the polyhedral representation. Section~4 gives the \(\mathcal{VU}\)-structure on that class, the transfer theorem, the approximate-leaf theorem, and the polyhedral enlargement. Section~5 records partial smoothness relative to the lineality space.

\section{Preliminaries}
Standard convex-analysis notation follows Rockafellar--Wets: \(\operatorname{co}\), \(\operatorname{aff}\), \(\operatorname{ri}\), \(\operatorname{int}\), \(\operatorname{cl}\), \(\operatorname{pos}\), \(\sigma_S\), \(\delta_S\), \(S^\circ\), \(S^\infty\), \(N_S\), \(T_S\), \(\partial\), \(\partial^\infty\), \(f'(x;d)\), \(\operatorname{Ker}\), \(\operatorname{Im}\), and the closed ball \(B(x,r)\).
We write \(\Gamma_0(\mathbb{R}^n)\) for the proper convex lsc functions, and \(\mathcal{H}\) for the proper closed sublinear ones (Definition~\ref{def:sublinear}).
The face of a set \(S\) exposed by \(d\) is \(F_S(d)\); \(\operatorname{ext} S\) denotes the extreme points and extreme directions.
The \(\mathcal{VU}\) subspaces at \(x\in\operatorname{dom}\partial f\) are \(\mathcal{V}f(x)\) and \(\mathcal{U}f(x)\) (Definition~\ref{def:vu}); \(\mathcal{R}_{h,F}\) is the exact chain-rule set of \cite[Definition~2]{Liu2026vu}.
The closed unit balls of \(\|\cdot\|_1\) and \(\|\cdot\|_\infty\) are \(\mathbb{B}_1\) and \(\mathbb{B}_\infty\); \(\Delta_n\) is the canonical simplex.
\(\mathcal{S}_n\) is the space of real symmetric \(n\times n\) matrices, and \(\mathcal{C}_n=\{V\succeq 0:\operatorname{tr} V=1\}\) is the spectraplex.
\begin{definition}[sublinear]\label{def:sublinear}
A function \(h:\mathbb{R}^n\to\overbar{\mathbb{R}}\) is \emph{sublinear} if it is subadditive and positively homogeneous of degree \(1\):
\[
h(x+y)\le h(x)+h(y),\qquad h(tx)=th(x)\quad\text{for all }t>0.
\]
For a proper function these two axioms are equivalent to convexity together with positive homogeneity of degree \(1\) \cite[Theorem~4.7]{Rockafellar1970}. The class \(\mathcal{H}\) is the proper closed members of this class.
\end{definition}

The \(\mathcal{VU}\) decomposition was initially defined in \cite{Lemarechal2000} for a finite-valued convex function and later was generalized to the setting of any proper, lsc functions in \cite[Definition 5]{Liu2020}. 
\begin{definition}[\(\mathcal{VU}\) decomposition ]\label{def:vu}
Given a proper, lsc function \(f\)  and a point \(\bar{x}\) with \(\partial f\mbk{\bar{x}}\not=\emptyset\), the \(\mathcal{VU}\) decomposition associated with \(\partial f\mbk{\bar x}\) is defined by 
\begin{equation}\label{eq:defVU}
\mathcal{V}f(\bar{x})\mathrel{\mathop:}= \spn(\partial f(\bar{x})-{g} ),\quad \mathcal{U}f(\bar{x})\mathrel{\mathop:}=\mathcal{V}f\mbk{\overbar{
x}}^\bot
\end{equation}
where \({g}\) is an arbitrary subgradient in \(\partial f(\bar{x})\). 

Accordingly, for any $x\in \mathbb{R}^n$
we have its two components associated with the 
decomposition via,
\begin{equation}
x_u\mathrel{\mathop:}=P_{\mathcal{U}f(\bar{x})}(x),\qquad x_v\mathrel{\mathop:}=P_{\mathcal{V}f(\bar{x})}(x).
\end{equation}   
\end{definition}
The  relative interior of the subdifferential \(\partial f(\bar{x})\) is defined by
\(\ri\partial f(\bar{x})\mathrel{\mathop:}=\mcb{g\in\mathbb{R}^n\colon g+\inte B(0,\eta)\cap \mathcal{V}\subset \partial f(\bar{x})\text{ for some }\eta>0}\).
If \(f\) is  subdifferentially regular at \(\bar{x}\), then according to \cite[Proposition 5]{Liu2020} (see also \cite[Lemma 12]{Hare2001quadratic}) the subspaces in question can also be characterized as
\begin{gather}
\mathcal{U}f(\bar{x})=N_{\partial f(\bar{x})}(g^\circ)=\mcb{w\in\mathbb{R}^n  \colon f'(\bar{x};-w)=
-f'\mbk{\bar{x};w }},\label{Uxdefalter}\\
\mathcal{V}f(\bar{x})=T_{\partial f(\bar{x})}(g^\circ),\label{Vxdefalter}
\end{gather}
When analyzing the chain rule of  \VU-decomposition we will need the validity of the standard subdifferential chain rule:
\begin{gather}\notag
\partial (h\circ F)(x) = F'(x)^\top \partial h(F(x)), \tag{CR}
\end{gather}
for a smooth mapping \(F\colon\mathbb{R}^n\to\mathbb{R}^m\) and a proper closed function \(h\colon\mathbb{R}^m\to\overbar{\mathbb{R}}\). 
A fundamental condition (see \cite[Theorem 10.6]{Rockafellar1998}) at a point \(x\in \dom (h\circ F)\) is 
\begin{equation}\label{eq:icqo} 
\Ker F'(x)^\top \cap \partial^\infty h\mbk{F(x)}=\mcb{0},
\end{equation} 
where \(\partial^\infty h(\cdot)\) is the horizon subdifferential. 
The set \(\mathcal{R}_{h,F}\) is \cite[Definition~2]{Liu2026vu}: for \(\bar x\in\dom(h\circ F)\),
\begin{enumerate}[label=(R\arabic*)]
\item \label{R1} \(F(\bar x)\in\ri(\dom h)\) and \(F\) maps a neighbourhood of \(\bar x\) into \(\aff(\dom h)\); or
\item \label{R2} \(F(\bar x)\in\dom h\setminus\ri(\dom h)\) and \eqref{eq:icqo} holds.
\end{enumerate}
This is a name for the two alternatives that give the convex chain rule, not a new qualification: (R2) is \cite[Theorem~10.6]{Rockafellar1998}, and (R1) is the relative-interior chain rule of \cite{Liu2026vu}. If \(\aff(\dom h)\neq\mathbb{R}^m\), then \(\partial^\infty h=(\aff(\dom h))^\perp\) throughout \(\ri(\dom h)\), so (R1) need not imply (R2). When \(h\) is finite-valued one has \(\partial^\infty h\equiv\{0\}\) and \(\mathcal{R}_{h,F}=\mathbb{R}^n\).
The chain rule (CR) holds at every point of \(\mathcal{R}_{h,F}\).
The one-sided set limits below use only nested inclusions; the next statement records that convention.
\begin{proposition}\label{prop:basicSemiC}
For a closed-valued mapping \(S \colon \mathbb{R}^n \rightrightarrows  \mathbb{R}^m\) and a point \(\bar x\in \mathbb{R}^n\),
if there exists a neighborhood \(\mathcal{N}\) of \(\bar{x}\) such that
\(S(x)\subset S(\bar{x})\) for all \(x\in \mathcal{N}\) then
\(\limsup_{x\rightarrow \bar x}S(x)=S(\bar{x})\).
On the other hand, if
\(S(\bar x)\subset S({x})\) for all \(x\in \mathcal{N}\) then
\(\liminf_{x\rightarrow \bar x}S(x)=S(\bar{x})\).
\end{proposition}
\begin{proof}
Suppose \(S(x)\subset S(\bar x)\) for all \(x\) near \(\bar x\), with \(S(\bar x)\) closed. Any limit of points \(y_k\in S(x_k)\) with \(x_k\to\bar x\) then lies in \(S(\bar x)\), so \(\limsup S(x)\subset S(\bar x)\). If in addition \(S(\bar x)\subset S(x)\) near \(\bar x\), the constant selection \(y_k\in S(\bar x)\subset S(x_k)\) gives the reverse inclusion, hence equality. The inner-limit statement is the same argument with the inclusions reversed.
\end{proof}
\begin{definition}[horizon outer limit]\label{def:horizon}
For a sequence  \(\{C^\nu\}_{\nu\in\mathbb N}\) 
of subsets of \(\mathbb R^n\), the \emph{horizon outer limit} is defined by
\begin{multline*}
\limsup\nolimits_{\nu}^\infty C^\nu
\mathrel{\mathop:}= \{0\}
\cup
\bigl\{x \colon \exists\text{ infinite } \mathcal{K}\subset \mathbb{ N},\ 
\exists\, x^\nu\in C^\nu,\\
\lambda^\nu \searrow 0, \ \nu\in \mathcal{K},\ 
\lambda^\nu x^\nu \urightarrow[\mathcal{K}]{} x
\bigr\}\,.
\end{multline*}
 \end{definition}
 \begin{lemma}\label{lem:horizonout}
 Let \(\{C^\nu\}_{\nu\in\mathbb N}\) be a sequence of cones in \(\mathbb R^n\). Then 
 \begin{equation}
 \limsup\nolimits_{\nu}^\infty C^\nu=\limsup\nolimits_{\nu} C^\nu\,.
 \end{equation}
 \end{lemma}
 \begin{proof}
Cones absorb the factors \(\lambda^\nu\searrow 0\) in the horizon outer limit, so the two limits coincide \cite[Theorem~4.25(b)]{Rockafellar1998}.
\end{proof}
 \begin{lemma}[Separation Criterion for Convex Sets and Affine Subspaces]
\label{thm:separation}
Let \(A \subset \mathbb{R}^n\) be a nonempty convex set, let \(y \in A\), let \(L \subset \mathbb{R}^n\) be a linear subspace, and set \(B := y + L\). Then 
\[
A \text{ and } B \text{ cannot be separated }\Longleftrightarrow
L^\perp  \cap N_A(y)  = \{0\}
\]
\end{lemma}
\begin{proof}
A proper separation of the convex set \(A\) from the affine set \(B\) produces a nonzero normal to \(A\) at \(y\) orthogonal to \(L\) \cite[Theorem~11.3]{Rockafellar1970}. The converse is the same supporting hyperplane. Theorem~6.9 of \cite{Rockafellar1998} is the tangent--normal relation for a convex set, not this criterion.
\end{proof}
The exact chain rule
\[
\mathcal{V}(h\circ F)(\bar x)=F'(\bar x)^\top\mathcal{V}h(F(\bar x)),
\qquad
\mathcal{U}(h\circ F)(\bar x)=F'(\bar x)^{-1}\bigl(\mathcal{U}h(F(\bar x))\bigr)
\]
under \(\bar x\in\mathcal{R}_{h,F}\), together with the separable, convex-sum, smooth-plus-lsc, model, and finite-max formulae, is \cite[Theorem~1 and Propositions~2--6]{Liu2026vu}.
We use that calculus as a black box: its value here is only to convert an outer \(\varepsilon\)-statement on \(h\) into one on \(f=h\circ F\).
\begin{assumption}\label{as:2}
Given \(h\in\mathcal{H}\) and a mapping \(F:\mathbb{R}^n\to\mathbb{R}^m\) of class \(C^1\), the point \(x\) lies in \(\mathcal{R}_{h,F}\): either \ref{R1} or \ref{R2} holds at \(x\).
\end{assumption}

\subsection{Enlargements and \texorpdfstring{\(\varepsilon\)-\(\mathcal{VU}\)}{epsilon-VU} spaces}\label{ssec:epsO}
The symbols special to this subsection are collected here.
\begin{description}[leftmargin=*, labelwidth=!,itemindent=!, align=left]
  \item[\(\partial_\varepsilon f(x)\)] the convex \(\varepsilon\)-subdifferential of \(f\) at \(x\)
  \item[\(\bar{\partial}_\varepsilon f(x)\)] a proper approximation of \(\partial f(x)\) contained in \(\partial_\varepsilon f(x)\)
  \item[\(\mathcal{V}_\varepsilon f(x)\), \(\mathcal{U}_\varepsilon f(x)\)] the \(\varepsilon\)-\(\mathcal{VU}\) subspaces associated with \(\bar{\partial}_\varepsilon f(x)\), Definition~\ref{def:epsVU}
  \item[\(\mathcal{M}_\varepsilon(x)\)] the approximate leaf \(F^{-1}\bigl(F(x)+\mathcal{U}_\varepsilon h(F(x))\bigr)\), Theorem~\ref{thm:Meps}
\end{description}
The standard convex \(\varepsilon\)-subdifferential is written \(\partial_\varepsilon h(x)\). An \emph{enlargement} \(\bar{\partial}_\varepsilon h(x)\) is a (possibly strictly smaller) set-valued approximation of \(\partial h(x)\) contained in \(\partial_\varepsilon h(x)\); the two objects are not interchangeable. 
\begin{definition}[Proper approximation of subdifferential {\cite[Definition~3.2]{Liu2026vu}}]\label{def:properapprox}
Let \(f\colon \mathbb{R}^n\rightarrow \lowoverline{\mathbb{R}}\) be proper and lsc, and let \(\bar{\partial}_\varepsilon f(x)\colon\mathbb{R}^n\times \mathbb{R}_{\geq 0}\rightrightarrows\mathbb{R}^n
\) be a convex-valued mapping such that 
\begin{gather}
\partial f\mbk{{x}}\subseteq \bar{\partial}_\varepsilon f(x)\text{ and } \bar{\partial}_0 f(x)=\partial f\mbk{x}\text{ for all }x\in \mathbb{R}^n\,.
\end{gather}
Then we call   \(\bar{\partial}_\varepsilon f(x)\) a proper approximation of \(\partial f\mbk{x}\)\,.
\end{definition}
\begin{definition}[\(\varepsilon\)-\VU-decomposition {\cite[Definition~3.3]{Liu2026vu}}]\label{def:epsVU}
Consider a proper, lsc function \(f\)  and a point \(\bar{x}\) with \(\partial f\mbk{\bar{x}}\not=\emptyset\). 
Let \(\bar{\partial}_\varepsilon f(x)\) be a proper approximation of \(\partial f\mbk{x}\). 
The \emph{\(\varepsilon\)-\VU-decomposition associated with \(\bar{\partial}_\varepsilon f(\bar x)\)} is defined by two subspaces
\begin{gather}
\mathcal{V}_\varepsilon f(x)\mathrel{\mathop:}=\spn (\bar{\partial}_\varepsilon f(x)-g),\text{ for an arbitrary }g\in \bar{\partial}_\varepsilon f(x)\,,
\label{def:vep}\\
\mathcal{U}_{\varepsilon}f(x)\mathrel{\mathop:}=N_{\bar{\partial}_\varepsilon f(x)}(g^\circ),\text{ for an arbitrary } g^\circ\in\ri \bar{\partial}_\varepsilon f(x)\,.\label{def:uep}
\end{gather}
\end{definition}
The following proposition shows that the \(\varepsilon\)-\VU-decomposition is indeed a  decomposition of \(\mathbb{R}^n\). 
\begin{proposition}[equivalent descriptions {\cite[Proposition~3.8]{Liu2026vu}}]\label{prop:uepsaltdef}
In Definition \ref{def:epsVU}, we have
\begin{gather}
\mathcal{U}_{\varepsilon}f(x)=\mcb{w\in\mathbb{R}^n\colon \mang{g-g^{\circ},w}=0\text{ for all } g\in\bar{\partial}_\varepsilon f(x)}\,,\ 
\forall g^\circ\in\ri \bar{\partial}_\varepsilon f(x)\,,\notag
 \\
\mathcal{U}_{\varepsilon}f(x)=\mathcal{V}_\varepsilon f(x)^{\perp}\label{eq:4}\,.
\end{gather}
\end{proposition}
The equivalent descriptions are \cite[Proposition~3.8]{Liu2026vu}; the proof is not repeated. Both spaces are independent of the auxiliary vectors, and $\mathbb{R}^n=\mathcal{V}_\varepsilon f(x)\oplus\mathcal{U}_\varepsilon f(x)$.

Now we would like to extend the results associated with the \(\varepsilon\)-objects to a composite function of the form  \(f(x)=h\mbk{F(x)}\) for a smooth mapping \(F\colon\mathbb{R}^n\to\mathbb{R}^m\) and a function \(h\in\mathcal{H}\). 
Suppose that \(\bar{\partial}_\varepsilon h(\cdot)\), a proper approximation of \(\partial h\mbk{\cdot}\) is available. 
Assumption \ref{as:2} is a condition at one point. At each \(x\in\dom f\) at which it holds and \(\bar{\partial}_\varepsilon h(F(x))\neq\emptyset\), the transported set \(F'(x)^\top\bar{\partial}_\varepsilon h(F(x))\) is the enlargement used below; it is not claimed to be a proper approximation on all of \(\dom f\). Write 
\begin{gather}\notag
\bar{\partial}_\varepsilon f(x)\mathrel{\mathop:}= F'(x)^\top\bar{\partial}_\varepsilon h(F(x)), \text{ for all }x\in \dom f\,.
\end{gather}
Then \(\mathcal{V}_\varepsilon f(x)\) and \(\mathcal{U}_{\varepsilon}f(x)\) can be defined from Definition \ref{def:epsVU} and satisfies Proposition \ref{prop:uepsaltdef}.
The chain rule for the enlargement is the same linear algebra as the exact one, with \(\partial h\) replaced by \(\bar\partial_\varepsilon h\).
\begin{proposition}\label{prop:epsVUchain}
Suppose \(f(x)=h\circ F(x)\) for a smooth mapping \(F:\mathbb{R}^n\to\mathbb{R}^m\) and \(h\in\Gamma_0(\mathbb{R}^m)\), and let \(\bar\partial_\varepsilon h\) be a proper approximation of \(\partial h\).
If \(x\in\mathcal{R}_{h,F}\) and \(\bar\partial_\varepsilon h(F(x))\neq\emptyset\), then
\begin{gather}
\mathcal{V}_\varepsilon f(x)=F'(x)^\top\mathcal{V}_\varepsilon h\bigl(F(x)\bigr),\label{eq:uvefcom}\\
\mathcal{U}_\varepsilon f(x)=\bigl\{y:F'(x)y\in\mathcal{U}_\varepsilon h\bigl(F(x)\bigr)\bigr\}.\notag
\end{gather}
This is \cite[Proposition~3.9]{Liu2026vu}.
\end{proposition}

\section{Properties of class \texorpdfstring{\(\mathcal{H}\)}{H}}
A function is sublinear in the sense of Definition~\ref{def:sublinear}: subadditive and positively homogeneous of degree \(1\), equivalently convex and positively homogeneous of degree \(1\). 
There is a one-to-one correspondence between functions in \( \mathcal{H}\) and the nonempty, closed, convex subsets \(D\) of \(\mathbb{R}^n\): 
\begin{equation}\notag
h \longleftrightarrow D\ \text{with }h=\sigma_D\text{ and }D=\mcb{v\colon \mang{v,\cdot}\leq h}\,.
\end{equation}
We summarize the basic properties of functions in class \(\mathcal{H}\). 
 \begin{proposition}\label{prop:sub}
If \(h\in\mathcal{H}\) then 
\begin{enumerate}[label=\textnormal{(\roman*)}]
\item \(h(0)=0\);\label{it1:1}
\item \(h(\cdot)	=\sigma_{\partial h	(0)}(\cdot)\) where \(\partial h(0)=\mcb{s\in\mathbb{R}^n	\colon h(x)\geq \mang{s,x}\ \forall x\in\mathbb{R}^n}\);\label{it1:2}
\item \(\partial h(x)=\mcb{s\in \partial h(0)\colon h(x)=\mang{s,x}}=F_{\partial h(0)}(x)\);\label{it:1-3}
\item \(\partial_\varepsilon h(x)=\mcb{s\in\partial h(0)\colon \mang{s,x}\geq h(x)-\varepsilon}\);\label{it:1-4}
\item \(\partial_\varepsilon h(x)\subset \partial h(0)=\partial_\varepsilon h(0), \ \fa \varepsilon\geq 0, x\in \mathbb{R}^n\);\label{it:1-5}
\item \(h\) is finite valued if and only if \(\partial h(0)\) is bounded.\label{it:1-8}
\item if  \(0\in \inte \dom h\)
then  \(\lim_{\substack{x\to 0\\\varepsilon\searrow \bar \varepsilon }}\partial_\varepsilon h(x)=\partial h(0)\) for any \(\bar\varepsilon>0\).\label{it:1-7}
\item \(h'(0;d)=h(d)\)\label{item:direcDir}
\item \(h\) has a global minimum at \(0\) if and only if \(h(x)\geq 0\) for all \(x\in \mathbb{R}^n\). \label{it:10}
\item \label{it:3111}\(h\) is piecewise linear if and only if \(\partial h(0)\) is polyhedral.
\end{enumerate}
\end{proposition}
\begin{proof}
Items (i)--(vi) are \cite[Corollary~8.25]{Rockafellar1998} and \cite[Ch.~VI, Example~3.1; Ch.~XI, Example~1.2.5]{HiriartUrruty1993a}. Item (xi) is \cite[Theorem~19.1]{Rockafellar1970}: a support function is polyhedral if and only if the set it supports is polyhedral.
Item (vii) is the continuity of \((x,\varepsilon)\mapsto\partial_\varepsilon h(x)\) on \(\operatorname{int}(\operatorname{dom} h)\times(0,\infty)\) \cite[Ch.~XI, pp.~128--129]{HiriartUrruty1993b}.
The remaining items are immediate from positive homogeneity and (ii).
\end{proof}
The exact and \(\varepsilon\)-\(\mathcal{VU}\) spaces are defined only on \(\dom\partial h\). For \(h\in\mathcal{H}\) that set is classical and needs no new letter.
\begin{proposition}\label{prop:domsub}
Let \(h\in\mathcal{H}\) and write \(D:=\partial h(0)\). Then
\begin{enumerate}[label=\textnormal{(\roman*)}]
\item \(0\in\dom\partial h\) and \(\ri(\dom h)\subset\dom\partial h\subset\dom h\);
\item \(\dom\partial h=\mcb{x\colon F_D(x)\neq\emptyset}\), the directions in which the closed convex set \(D\) admits a supporting hyperplane;
\item if \(h\) is finite-valued, equivalently if \(D\) is bounded, then \(\dom\partial h=\mathbb{R}^n\);
\item if \(h\) is polyhedral then \(\dom\partial h=\dom h\).
\end{enumerate}
\end{proposition}
\begin{proof}
Item (i) is \cite[Theorem~23.4]{Rockafellar1970} together with \(h(0)=0\). Item (ii) is Proposition~\ref{prop:sub}\ref{it:1-3}. Item (iii) is Proposition~\ref{prop:sub}\ref{it:1-8}: a compact convex set is supported in every direction. Item (iv) is the subdifferentiability of a polyhedral convex function throughout its domain \cite[Theorem~23.10]{Rockafellar1970}.
\end{proof}
The next example shows that a supporting face may fail on the boundary of \(\dom h\).
\begin{example}\label{eg:emptysub}
A supporting face may fail on the boundary of \(\dom h\).
Let \(D=\operatorname{epi}(\exp)\subset\mathbb{R}^2\) and set \(h=\sigma_D\).
For every \(y<0\) one has \(h(0,y)=0\), so the closed ray
\(\{0\}\times(-\infty,0]\) lies in \(\dom h\).
The supremum \(\sup_{s\in\mathbb{R}} y e^{s}\) equals \(0\) and is not attained, hence
\((0,-1)\in\dom h\setminus\dom\partial h\).
\end{example}

The class \(\mathcal{H}\) is closed under the operations below.
\begin{proposition}[Operations preserving sublinearity]\label{prop:preserve}
The following operations preserve sublinearity. 
\begin{enumerate}[label=(\alph*)]
\item \(h(x)=t_1\sigma_{S_1}(x)+t_2\sigma_{S_2}(x)=\sigma_S(x)\) where \(t_i> 0,\ i=1,2\), \(S_i\) is nonempty convex and 
\(S\mathrel{\mathop:}=t_1S_1+t_2S_2\).\label{it:321}
\item{Separability}
Let \(f(x)=\sum_{i=1}^mf_i(x_i)\) for proper closed sublinear functions \(f_i\colon\mathbb{R}^{n_i}\rightarrow \overbar{\mathbb{R}}\) where \(x\in  \mathbb{R}^n\) is expressed as \(x_1,\cdots,x_m\) with \(x_i\in\mathbb{R}^{n_i}\). 
Then \(f(x)\in \mathcal{H}\).  
\end{enumerate}
\end{proposition}
\begin{proof}
Item (a) is \cite[Corollary 11.24]{Rockafellar1998}. For (b), \(f(t_1x+t_2y)=\sum f_i(t_1x_i+t_2y_i)\leq t_1f(x)+t_2f(y)\), and \(f\) is the support function of \(S_1\times\cdots\times S_m\).
\end{proof}

\begin{lemma}\label{lem:2}
Let \(h\in\mathcal{H}\) and \(x\in\mathbb{R}^n\). 
If \(d\) is such that 
\begin{equation}\label{persist}
\partial h(x)\cap \partial h(x+d) \not=\emptyset
\end{equation}
then
\[h(x+d)=h(x)+h'(x;d)\,.\]
\end{lemma}
\begin{proof}
The inequality $h(x+d)\geq h(x)+h'(x;d)$ holds by convexity of $h$, using that $h'(x;d)=\displaystyle\sup_{s\in\partial h(x)}\left<s,d\right>$.
To show the reverse inequality notice that, for any $s\in\partial h(x)$,
\begin{eqnarray*}
h(x)&\geq&\left<s,x\right>\\
h'(x;d)&\geq&\left<s,d\right>\,.
\end{eqnarray*}

Adding both lines and taking \(s'\in \partial h(x)\cap \partial h(x+d) \) we see that $h(x)+h'(x;d)\geq \left<s',x+d\right>=h(x+d)$ for all such $s'$. 
\end{proof}

\subsection{Examples}
The examples below record the sublinear functions and the composites with sublinear outer function used later.
\begin{example}[Sublinear functions as support functions]\label{eg:subli}
From Proposition \ref{prop:sub}\ref{it1:2} we see that if \(h\in\mathcal{H}\) then it can be written as the support function of \(\partial h(0)\). 
A list of functions with such form follows.
\begin{enumerate}[label=\textnormal{(\roman*)}]
\item \(h(x)=\max_{1\leq i\leq n}x_i=\sigma_{\Delta ^n}(x)\). It is known that
\begin{gather*}
\partial h(x)=F_{\Delta_n}(x)\\
=\mcb{\alpha\in\Delta_n\colon\mang{\alpha,x}\geq h(x)}\\
=\mcb{\alpha\in\Delta_n\colon \alpha_i=0\text{ if }x_i<h(x)}\,.
\end{gather*}
\item \(p\)-norms \(h(x)=\normt{x}_p=\sigma_{\partial h(0)}(x)\) with
\[
\partial h(0)=\mcb{y\in \mathbb{R}^n\colon \normt{y}_q\leq 1, \ \tfrac{1}{p}+\tfrac{1}{q}=1,\ p\geq 1}\,.
\]
\item \(\ell_{1,2}\)-norm. \( h(x)=\normt{x}_{1,2}=\sum_{i\in I}\normt{x_i}\), 
\(x_i\in \mathbb{R}^{n_i},\partial h(0)=\Pi_{i=1}^m\mathbb{B}^{n_i}\), where \(\Pi_i\mathbb{R}^{n_i}=\mathbb{R}^n\).
\item \(h(x)=\delta_K(x)=\sigma_{K^\circ}(x)\), where \(K\) is a nonempty closed convex cone. Cone indicators, gauges \(\gamma_C\), and \(\sqrt{f}\) for nonnegative degree-2 homogeneous \(f\) are also support functions \cite[Chapter~V]{HiriartUrruty1993a}; gauges are treated in \cite{Vaiter2015Model}.
\item The largest eigenvalue \(\lambda_1\) on \(\mathcal{S}_n\) (notation list) is the support function of the spectraplex \(\mathcal{C}_n\), so \(\partial\lambda_1(0)=\mathcal{C}_n\).
\label{eg:largestEigen}
It is known that
\[
\partial \lambda_1(X)=\co \mcb{qq\top\colon\normt{q}=1,\, q\in E_1(X)}
\]
where \(E_1(X)\) is the eigenspace associated with \(\lambda_1(X)\). 
\item Nuclear norm.
The map \(A\mapsto\|A\|_*\) is finite-valued and sublinear on \(\mathbb{R}^{n\times m}\),
with \(\partial\|\cdot\|_*(0)\) equal to the spectral-norm unit ball, which has nonempty interior. It is partly smooth relative to the constant-rank manifolds \cite[Table~1]{doi:10.1137/16M106340X}.
\end{enumerate}
\end{example}
\begin{example}[\(h(x)=\abs{x}\)]\label{eg:1}
The absolute value function \(h\) has \(\partial h(0)=[-1,1]\) and applying Proposition \ref{prop:sub}\ref{it1:2} we see 
\(h(x)=\sigma_{[-1,1]}(x)\). 
The \(\varepsilon\)-subdifferential is 
\begin{equation}\label{eq:eps-sub}
\partial_{\varepsilon}h(x) =\begin{cases}
[-1,-1-\varepsilon/x] &\text{ if }x<-\varepsilon/2,\\
[-1,1] &\text{ if }-\varepsilon/2\leq x\leq \varepsilon/2,\\
[1-\varepsilon/x,1] &\text{ if }x>\varepsilon/2.
\end{cases}
\end{equation}
\end{example}

\begin{remark}[the simultaneous limit fails]\label{rem:simul}
The \(\bar\varepsilon\) in Proposition~\ref{prop:sub}\ref{it:1-7} cannot be taken as \(0\).
For \(h=\lvert\cdot\rvert\), the formula \eqref{eq:eps-sub} gives \(\partial_{\varepsilon}h(\varepsilon)=[0,1]\) when \(\varepsilon>0\).
The inner-limit test of \cite[Theorem~4.10(a)]{Rockafellar1998} would require \([-1,1]\cap\rho\mathbb{B}\subset[0,1]+\eta\mathbb{B}\) for every \(\rho,\eta>0\) and all small \(\varepsilon\).
Taking \(\rho=1\) and \(\eta=\tfrac12\) puts \(-1\) outside \([-\tfrac12,\tfrac32]\).
Hence \(\liminf_{x\to 0,\,\varepsilon\searrow 0}\partial_\varepsilon h(x)\neq\partial h(0)\).
\end{remark}

\begin{example}[group lasso]\label{eg:LW}
Let \(g_i:\mathbb{R}^n\to\mathbb{R}^{m_i}\) be of class \(C^1\), \(i=1,\ldots,t\), and set
\[
F(x)=(g_1(x),\ldots,g_t(x)),\qquad
h(y_1,\ldots,y_t)=\sum_{i=1}^t|y_i|.
\]
Then \(h\in\mathcal{H}\) is the \(\ell_{1,2}\) norm of Example~\ref{eg:subli}: each \(|y_i|\) is the Euclidean norm of one group, and \(\partial h(0)\) is the product of the corresponding unit balls.
The composite \(h\circ F\) penalises groups of coefficients together, which is the group-lasso regulariser.
It is not polyhedral.
\end{example}
\subsection{Minimizers of functions in \texorpdfstring{\(\mathcal{H}\)}{H}}
For  \(h\in\mathcal{H}\), let 
\begin{equation}
\label{eq:argminH}
W^h\mathrel{\mathop:}=\argmin_{x\in\mathbb{R}^n} h(x).\end{equation} 
\begin{proposition}\label{prop:end}
If \(W^h\not=\emptyset \) then \(0\in W^h\) and \(\min h=0\). 

\(
W^h\not=\emptyset \Leftrightarrow h(x)\geq 0, \ \fa x\in\mathbb{R}^n.
\) 
\end{proposition}
\begin{proof}
If \(W^h\not=\emptyset \) then there must be a point \(x^*\in \mathbb{R}^n\) such that \(0\in \partial h(x^*)\). 
From the expression of \(\partial h\) in Proposition \ref{prop:sub}\ref{it:1-3} we get \(0\in\partial h(0)\) and thus \(0\in W^h\). And from Proposition \ref{prop:sub}\ref{it1:1} we have \(\min h=0\).

The assertion that \(W^h\not=\emptyset\) under the nonnegativity of \(h\) can be obtained from Proposition \ref{prop:sub}\ref{it:10}. 
Suppose \(W^h\not=\emptyset\). Then from the convexity of \(h\) there exists \(x^*\in\mathbb{R}^n\) such that \(0\in\partial h(x^*)=\mcb{s\in \partial h(0)\colon\mang{s,x^*}=h(x^*)}\), and therefore \(0\in\partial h(0)\) and \(0=h(x^*)\). 
The convexity of \(h\) implies \(h(x)\geq h(x^*)+\mang{0,x-x^*}=0\), for any \(x\in\mathbb{R}^n\).  
\end{proof}
The optimal solution set is not necessarily a singleton. 
Consider the following example. 
\begin{example}\label{pos:semi}
{\(h(x)=\sqrt{\mang{Qx,x}}\)}
where \(Q\) is a symmetric positive semi-definite operator. 
By \cite[Chapter V, Example 1.2.3 and Definition 1.2.4]{HiriartUrruty1993a} \(h(x)=\gamma_{{E_Q}}(x)\) where \\\(E_Q\mathrel{\mathop:}=\mcb{x\in \mathbb{R}^n\colon \mang{Qx,x}\leq 1}\). 
By \cite[Chapter V, Proposition 3.2.4]{HiriartUrruty1993a} \(\gamma_{{E_Q}}(x)=   \sigma_{{E_Q}^\circ}(x)\). 
The set of optimal solutions is \(W^h=\mcb{y\colon y^\top Qy=0}\).
\end{example}
The minimizer set is a convex cone.
\begin{corollary}[cone property of the minimizer set]\label{cor:Wh-cone}
If \(x^*\in W^h\), then 
\begin{gather*}
cx^*\in W^h, \text{ if }c\geq 0 \\
x^*+y^*\in W^h, \text{ if } y^*\in W^h
\end{gather*}
\end{corollary}
\begin{proof}
Combining Proposition \ref{prop:sub}\ref{it:10} and Proposition \ref{prop:end} we see that if the problem \(\min_{x\in\mathbb{R}^n} h(x)\) has optimal solutions then 0 must be one and \(0=\min_{x\in\mathbb{R}^n} h(x)\). From the positive homogeneity and subadditivity of \(h\) we get \(h(cx^*)=ch(x^*)=0\) and \(h(x^*+y^*)\leq h(x^*)+h(y^*)=0\).
\end{proof}

\subsection{Polyhedral functions in \texorpdfstring{\(\mathcal{H}\)}{H}}
Polyhedral functions are functions with polyhedral epigraphs. They are equivalent to convex piecewise linear functions\cite[Exercise 3.54]{Rockafellar1998}.
The exact \(\mathcal{VU}\) continuity for polyhedral enlargements is
\cite{Liu2019Subdifferential}.
What is recorded below is a specialization of that account to the subclass
\(\mathcal{H}\), together with the composite transfer
\(f=h\circ F\) when the outer function is polyhedral and sublinear.
\begin{proposition}\label{prop: polyheH}
For \(h\in\mathcal{H}\), \(h\) is polyhedral if and only if there exist \(a^i\in \mathbb{R}^n\), \(i\in I\) and \(c^j\in \mathbb{R}^n\), \(j
\in J\)  and 
\begin{equation}\label{eq:reprPolyh}
h(x)=\begin{cases}
\max_{i\in I}\mang{a^i,x} & \text{if }\mang{c^j,x}\leq 0,\ j\in J\\
+\infty & \text{otherwise}
\end{cases}
%
%
\end{equation}
where \(I\) and \(J\) are finite index sets such that \(I\cup J\not=\emptyset\), and the maximum over an empty index set is taken to be \(0\). Thus \(I=\emptyset\) recovers the indicator of a polyhedral cone. Let the ``active" index sets be defined as \(I(x)=\mcb{i\in I\colon \mang{a^i,x}=h(x)}\) and \(J(x)=\mcb{j\in J\colon\mang{c^j,x}=0}\). 
Then  we have
\begin{equation}\label{eq:polyex}
\partial h(x)=\operatorname{co}\{a^i:i\in I(x)\}+\operatorname{pos}\{c^j:j\in J(x)\},
\end{equation}
with the convention \(\operatorname{co}\emptyset=\{0\}\). Equivalently, if \(I(x)\neq\emptyset\),
\[
\partial h(x)=\mcb{\sum_{i\in I(x)}\alpha_ia^i+\sum_{j\in J(x)}\beta_jc^j\colon\sum_{i\in I(x)}\alpha_i=1,\ \alpha_i\geq 0,\ \beta_j\geq0},
\]
while if \(I(x)=\emptyset\) one has \(\partial h(x)=N_C(x)=\operatorname{pos}\{c^j:j\in J(x)\}\).
\end{proposition}
\begin{proof}
All functions in \(h\in\mathcal{H}\) are convex. 
Therefore by Proposition \ref{prop:sub}\ref{it1:2} and \ref{it:3111} we have 
 \(h\in\mathcal{H}\) is polyhedral if and only if \( \partial h(0)\) is polyhedral. 
 By \cite[Theorem~19.1]{Rockafellar1970}, \( \partial h(0)\) is polyhedral if and only if it can be represented in the form 
 \( \partial h(0)=\co\mcb{a^i\colon i\in I}+\pos\mcb{c^j\colon j\in J}\). 
 Since \(\partial h(0)\not=\emptyset\) it follows that \(I\cup J\not=\emptyset\).
 If \(I=\emptyset\) then \(\partial h(0)\) is a polyhedral cone, \(h=\delta_{(\partial h(0))^\circ}\), and \eqref{eq:reprPolyh} holds with the empty-maximum convention. 
 If \(I\neq\emptyset\), define \(S_1=\co\mcb{a^i\colon i\in I}\) and \(S_2=\pos\mcb{c^j\colon j\in J}\). 
 In view of Proposition \ref{prop:preserve}\ref{it:321}, we get 
 \(h(x)=\sigma_{\partial h(0)}(x)=\sigma_{S_1}(x)+\sigma_{S_2}(x)\). 
 From \cite[Ch. V, Example 3.4.1]{HiriartUrruty1993a} we have \(\sigma_{S_1}(x)=\max_{i\in I}\mang{a^i,x}\). 
 As \(S_2\) is a polyhedral cone  we have \(\sigma_{{S_2}}(x)=\delta_{{S_2}^\circ}(x)\) by \cite[Ch. V, Example 2.3.1]{HiriartUrruty1993a}, and \({S_2}^\circ=\mcb{y\in\mathbb{R}^n\colon\mang{c^j,y}\leq 0,j\in J}\) by \cite[Lemma 6.45]{Rockafellar1998}. 
 Consequently, \(h(x)=\max_{i\in I}\mang{a^i,x}+\delta_{{S_2}^\circ}(x)\) and the conclusion holds. 
If \(I=\emptyset\) then \(h=\delta_C\) and \eqref{eq:polyex} is the usual formula for \(N_C\). If \(I\neq\emptyset\), define \(h_1\mathrel{\mathop:}=\max_{i\in I}\mang{a^i,x}\) and clearly \(\partial h_1(x)=\mcb{\sum_{i\in I(x)}\alpha_ia^i \colon\sum_{i\in I(x)}\alpha_i=1,\ \alpha_i\geq 0, {i\in I(x)}} 
\). Define \(S\mathrel{\mathop:}=\mcb{x\in \mathbb{R}^n\colon \mang{c^j,x}\leq 0,\ j\in J}\) and \(h_2\mathrel{\mathop:}=\delta_S(x)\). Then \(\partial h_2(x)=N_S(x)=\mcb{
\sum_{j\in J(x)}\beta_jc^j\colon \beta_j\geq0,{j\in J(x)}}\)\,. As the function \(h_1\) is finite-valued, hence locally Lipschitz, we have \(\partial h(x)=\partial h_1(x)+\partial h_2(x)\) which is just \eqref{eq:polyex}\,.
\end{proof}
The representation \eqref{eq:reprPolyh} and the active sets \(I(x)\), \(J(x)\) are the standing notation below; they are not redefined.
From \eqref{eq:sub-con0} and \eqref{eq:scin0} the one-sided limits of Proposition~\ref{prop:3} are stated on the lineality space. For polyhedral \(h\in\mathcal{H}\) they hold at every point of \(\dom h\).
\begin{proposition}\label{prop:polyheiHp}
Let \(h\in\mathcal{H}\) be polyhedral and \(\bar x\in\dom h\). Then there exists a neighborhood \(B(\bar{x},\epsilon)\) such that for all \(x\in B(\bar{x},\epsilon)\cap\dom h\),
\begin{enumerate}[label=\textnormal{(\roman*)}]
\item 
\(\partial h(x)\subset\partial h(\bar{x})\)  
\item If \(x\) and \(x+u\) both lie in \(B(\bar x,\epsilon)\cap\dom h\) and \(I(x+u)=I(x)\), \(J(x+u)=J(x)\), then \(\partial h(x)=\partial h(x+u)\neq\emptyset\) and
\(h(x+u)=h(x)+\mang{g_0,u}\) for every \(g_0\in\partial h(x)\).
\end{enumerate}
\end{proposition}
\begin{proof}
From the continuity of the linear terms \(\mang{a^i,x}\) and \(\mang{c^j,x}\) in the definition of \(h\), we see that for any \(\bar{x}\), there exists a neighborhood \(\mathcal{N}\) of \(\bar{x}\) such that \(I(x)\subset I(\bar{x})\) and \(J(x)\subset J(\bar{x})\) for all \(x\in \mathcal{N}\). 
If \(I(x+u)=I(x)\) and \(J(x+u)=J(x)\), the representation \eqref{eq:polyex} gives \(\partial h(x)=\partial h(x+u)\). Every \(s\in\partial h(x)\) then exposes both \(x\) and \(x+u\), so \(h(x+u)-h(x)=\mang{s,u}\). The right-hand side is independent of the choice of \(s\) in the face, and equals \(\mang{g_0,u}\) for any \(g_0\in\partial h(x)\).
\end{proof}
The condition of Lemma~\ref{lem:2} holds for polyhedral functions in \(\mathcal{H}\) at every \(x\in\dom h\) at which the active sets \(I(\cdot)\) and \(J(\cdot)\) are locally constant, by Proposition~\ref{prop:polyheiHp}.
\section{\texorpdfstring{\(\mathcal{VU} \) }{VU} decomposition for class \texorpdfstring{\(\mathcal{H}\)}{H}}
This section records the \(\mathcal{VU}\)-structure of \(h\in\mathcal{H}\) and its transfer to \(f=h\circ F\). Throughout, \(h\) is proper, closed and sublinear, and \(F\) is of class \(C^1\). Nonconvex outer functions and merely Lipschitz inner mappings are left aside.
\subsection{Properties under \texorpdfstring{\(\mathcal{VU}\)}{VU}-decomposition}
For \(h\in\mathcal{H}\) the point lies in its own \(\mathcal{U}\)-space, and the subspaces are invariant along the lineality space.
\begin{proposition}\label{prop:xinU}
Let \(h\in\mathcal{H}\) and \(x\in\operatorname{dom}\partial h\). Then \(x\in\mathcal{U}h(x)\).
\end{proposition}
\begin{proof}
By Proposition~\ref{prop:sub}\ref{it:1-3}, every \(s\in\partial h(x)\) satisfies \(\mang{s,x}=h(x)\). Hence \(\mang{s-g,x}=0\) for all \(s,g\in\partial h(x)\), so \(x\) is orthogonal to \(\mathcal{V}h(x)=\operatorname{span}(\partial h(x)-g)\).
\end{proof}
Sublinearity is essential. The shifted absolute value \(x\mapsto\lvert x-1\rvert\) is convex and polyhedral but not sublinear, and at the kink one has \(\mathcal{U}(\lvert\cdot-1\rvert)(1)=\{0\}\), which does not contain \(1\). For the cone indicator \(h=\delta_{\mathbb{R}_+}\), \(\mathcal{U}h(t)=\mathbb{R}\) when \(t>0\) and \(\mathcal{U}h(0)=\{0\}\).
\begin{proposition}[lineality trivial if and only if the subdifferential has interior]\label{prop:U0int}
For \(h\in\mathcal{H}\), 
\(\mathcal{U} h(0)=\mcb{0} \Leftrightarrow \inte \partial h(0)\not=\emptyset\)
\end{proposition}
\begin{proof}
By the alternative definition of the \U space in \eqref{Uxdefalter},
\[
\mathcal{U} h(0)=\mcb{w\in \mathbb{R}^n\colon h'(0;-w)=-h'(0;w)}.
\]
Proposition \ref{prop:sub}\ref{item:direcDir} then gives
\(\mathcal{U} h(0)=\mcb{w\in \mathbb{R}^n\colon h(-w)=-h(w)}\).
As the support function of \(\partial h(0)\), \(h\) satisfies
\cite[Proposition 8.29(a)]{Rockafellar1998}: a convex set \(C\) has nonempty interior if and only if no nonzero \(v\) obeys \(\sigma_C(v)=-\sigma_C(-v)\).
\end{proof}
For the nuclear norm of Example~\ref{eg:subli}, the spectral-norm unit ball has nonempty interior, so Proposition~\ref{prop:U0int} gives \(\mathcal{U}\|\cdot\|_*(0)=\{0\}\).
\begin{proposition}\label{prop:2}
For \(h\in\mathcal{H}\) and \(x\in\mathbb{R}^n\), if \(u\in \mathcal{U}h(0)\), then
\begin{gather}
h(x+u)=h(x)+\mang{g_0,u	}, \ \fa g_0\in\partial h(0),\label{eq:hxu} \\
\partial h(x+u)=\partial h(x).\label{eq:ve}
\end{gather}
If in addition \(\partial h(x)\neq\emptyset\), then \(\partial h(x+u)\neq\emptyset\) and
\begin{gather}
h'(x+u;d)=h'(x;d),\ \fa d\in\mathbb{R}^n,\\
\mathcal{V}h(x+u)=\mathcal{V}h(x),\ \mathcal{U}h(x+u)=\mathcal{U}h(x).\label{eq:ue}
\end{gather}
If moreover \(0\in\partial h(0)\), then \(h(u)=0\).
In particular, \(x\in\dom\partial h\) if and only if \(x+u\in\dom\partial h\).
\end{proposition}
\begin{proof}
Take $g_0\in\partial h(0)$ and let
\[S\mathrel{\mathop:}=\partial h(0)-g_0\]
Since $h(x)=\sigma_{\partial h(0)}(x)=\max\{\left<g,x\right>:g\in\partial h(0)\}$, the identity
\begin{equation}\label{h-g0}
h(x)=\left<g_0,x\right>+\sigma_{S}(x)\end{equation}
holds for all $x\in\mathbb{R}^n$.
For any $x\in\mathbb{R}^n$ and $u\in \mathcal{U}h(0)$ it holds                 
that $\sigma_S(x+u)=\sigma_S(x)$ because \(S\subset\mathcal{V}h(0)\),
and, in view of \eqref{h-g0}, we have \eqref{eq:hxu}. 

To show \eqref{eq:ve}, by Proposition \ref{prop:sub}\ref{it:1-3} and \eqref{h-g0}, it suffices to show for all \(s\in\partial h(0)\), \(\mang{s,x+u}=\mang{g_0,x+u}+\delta_S(x+u)\) if and only if \(\mang{s,x}=\mang{g_0,x}+\delta_S(x)\). 
This is indeed true because $\sigma_S(x+u)=\sigma_S(x)$ and \(s-g_0\in\mathcal{U}h(0)^\perp\). 
If \(\partial h(x)\neq\emptyset\), then \eqref{eq:ve} forces \(\partial h(x+u)\neq\emptyset\), and
$h'(x;d)=\sup_{s\in\partial h(x)}\mang{s,d}= h'(x+u;d)$ for any $d\in\mathbb{R}^n$.
The identity \eqref{eq:ue} is then the definition \eqref{eq:defVU} applied to the same nonempty set.

Setting \(x=0\) in \eqref{eq:ve} we get  \(\partial h(u)=\partial h(0)\) for all \(u\in\mathcal{U}h(0)\), and in view of Proposition \ref{prop:sub}\ref{it:1-3} we have
\begin{equation}\label{eq:13}
 \mang{s,u}=h(u),\ \fa s\in \partial h(0). 
\end{equation}
 If \(0\in\partial h(0)\), then by \eqref{eq:13} we get \(h(u)=0\).
\end{proof}
For the set of minimizers of \(h\) defined in \eqref{eq:argminH} we have the following. 
\begin{proposition}[lineality contained in the minimizer set]\label{prop:UinWh}
If \(W^h\not=\emptyset \) then \(\mathcal{U}h(0)\subset W^h\). 
\end{proposition}
\begin{proof}
If \(W^h\not=\emptyset \) then from Proposition \ref{prop:end} we have 
\(0\in W^h\) and hence \( 0\in \partial h(0)\). 
In view of \eqref{eq:ve} with 0 in place of  \(x\) we have 
\(\partial h(u)=\partial h(0)\) for all \(u\in \mathcal{U}h(0)\). 
Consequently, \(0\in \partial h(u)\) for all \(u\in \mathcal{U}h(0)\) and the proof is finished. 
\end{proof}
Example~\ref{pos:semi} attains equality: for \(h(x)=\sqrt{\mang{Qx,x}}\) with \(Q\) positive semidefinite, \(W^h=\ker Q=\mathcal{U}h(0)\).
The subspaces are nested, and on the lineality space the limits are one-sided.
\begin{proposition}\label{prop:3}
Given \(h\in\mathcal{H}\) and \(x\in\mathbb{R}^n\) with \(\partial h(x)\neq\emptyset\),
\begin{gather}
\mathcal{V}h(x)\subset \mathcal{V}h(0), \ \mathcal{U}h(0)\subset  \mathcal{U}h(x).\notag
\end{gather}
If \(u\in \mathcal{U}h(0)\) then \(\partial h(u)=\partial h(0)\neq\emptyset\) and
\begin{gather}
\mathcal{V}h(u)=  \mathcal{V}h(0), \ \mathcal{U}h(u)= \mathcal{U}h(0),\label{eq:uv-eq}
\\
\limsup_{\substack{x\rightarrow u\\\partial h(x)\neq\emptyset}} \mathcal{V}h(x)=\mathcal{V}h(u),\label{eq:sub-con0}
\\
\liminf_{\substack{x\rightarrow u\\\partial h(x)\neq\emptyset}} \mathcal{U}h(x)=\mathcal{U}h(u).\label{eq:scin0}
\end{gather}
\end{proposition}
\begin{proof}
From Proposition \ref{prop:sub}\ref{it:1-3} we see \(\partial h(x)\) is a subset of \(\partial h(0)\) whenever it is nonempty. Definition \eqref{eq:defVU} then gives \(\mathcal{V}h(x)\subset\mathcal{V}h(0)\), and the orthogonal complement yields \(\mathcal{U}h(0)\subset \mathcal{U}h(x)\).
To see \eqref{eq:uv-eq}, substitute \(x=0\) in \eqref{eq:ue}; here \(\partial h(0)\neq\emptyset\) by Proposition \ref{prop:sub}\ref{it1:2}.

Consequently, \(\mathcal{V} h(x)\subset \mathcal{V} h(u) \) and \(\mathcal{U} h(u)\subset \mathcal{U} h(x)\) hold at every \(x\) with \(\partial h(x)\neq\emptyset\).
Proposition \ref{prop:basicSemiC} then gives \eqref{eq:sub-con0} and \eqref{eq:scin0} along that set. 

\end{proof}
The next example shows that the inclusion of Proposition~\ref{prop:3} is strict off the origin, and that Proposition~\ref{prop:2} is read on the rays.
\begin{example}[ReLU]\label{eg:relu}
The rectifier \(h(t)=\max\{t,0\}=\sigma_{[0,1]}(t)\) lies in \(\mathcal{H}\).
Then \(\partial h(0)=[0,1]\), so \(\mathcal{U}h(0)=\{0\}\) and \(\mathcal{V}h(0)=\mathbb{R}\).
The exact subdifferential is the exposed face \(\partial h(t)=\{1\}\) for \(t>0\), \(\{0\}\) for \(t<0\), and \([0,1]\) at \(t=0\).
Hence \(\mathcal{U}h(t)=\mathbb{R}\) on either open ray and \(\mathcal{U}h(0)=\{0\}\): the inclusion of Proposition~\ref{prop:3} is strict off the origin.
On the lineality space the identity is trivial, since \(\mathcal{U}h(0)=\{0\}\); Proposition~\ref{prop:2} is read on the rays, where \(\partial h\) is a singleton and the function is linear.
The \(\varepsilon\)-subdifferential is the slice
\[
\partial_\varepsilon h(t)
=
\{s\in[0,1]:st\ge h(t)-\varepsilon\}
=
\begin{cases}
[1-\varepsilon/t,\,1] & t>\varepsilon,\\
[0,1] & \lvert t\rvert\le\varepsilon,\\
[0,\,\varepsilon/\lvert t\rvert] & t<-\varepsilon.
\end{cases}
\]
\end{example}

\begin{remark}[the standard enlargement misses the smooth rays]\label{rem:relu-miss}
For the rectifier, \(\partial_\varepsilon h(t)\) has positive length for every \(\varepsilon>0\), so \(\operatorname{aff}\partial_\varepsilon h(t)=\mathbb{R}\) and \(\mathcal{U}_\varepsilon h(t)=\{0\}\) at every \(t\).
The identification \(\mathcal{U}_\varepsilon h(t)=\mathbb{R}\) on the open rays fails as soon as \(\varepsilon>0\).
The same sequences as in Remark~\ref{rem:simul} show that \(\varepsilon\searrow 0\) cannot be taken simultaneously with \(t\to 0\) at an arbitrary rate.
\end{remark}

For the group lasso of Example~\ref{eg:LW}, at a point where one group is nonzero and the others vanish, \(\partial h\) is a singleton in the active groups and a ball in the inactive ones, so \(\mathcal{U}h(F(x))\) strictly contains \(\mathcal{U}h(0)=\{0\}\): the nesting of Proposition~\ref{prop:3} is strict.

On the leaf the composite is affine in the outer variable.
\begin{proposition}\label{prop:compvu}
Suppose \(f(x)=h\mbk{F(x)}\) for a smooth mapping \(F\colon\mathbb{R}^n\to\mathbb{R}^m\) and a function \(h\in\mathcal{H}\). 
If \(x\in  F\inv\mbk{\mathcal{U}h(0)}\), then 
\begin{equation}\label{eq:fsmoF}
f(x)=\mang{g_0,F(x)}, \fa g_0\in\ri\partial h(0).
\end{equation}
If in addition \(x\) satisfies Assumption \ref{as:2}, then 
\begin{gather}
\partial f(x)=F'(x)^\top\partial h(0)
\label{eq:pfyinu}\\
f'(x;d)=h\mbk{F'(x)d},\ \fa d\in\mathbb{R}^n
\notag\\
\mathcal{V}f(x)=F'(x)^\top\mathcal{V}h(0),\qquad
\mathcal{U}f(x)=F'(x)^{-1}\bigl(\mathcal{U}h(0)\bigr)\label{eq:uvy}.
\end{gather}
\end{proposition}
\begin{proof}
Substitute \(x\) in Proposition \ref{prop:2} with 0 to see for all \(u\in\mathcal{U}h(0)\)
\begin{gather}
h(u)=h(0)+\mang{g_0,u}=\mang{g_0,u},\ \partial h(u)=\partial h(0),\label{eq:hu0}\\
h'(u;d)=h'(0;d)=h(d), \ \fa d\in\mathbb{R}^n,\ \mathcal{V}h(u)=\mathcal{V}h(0),\ \mathcal{U}h(u)=\mathcal{U}h(0).\label{eq:hdirec}
\end{gather}
If \(x\in  F\inv\mbk{\mathcal{U}h(0)}\), then  \(F(x)\in \mathcal{U}h(0)\), and by \eqref{eq:hu0} and \eqref{eq:hdirec} we have \(f(x)=h\mbk{F(x)}=\mang{g_0,F(x)}\), \(\partial h\mbk{F(x)}=\partial h(0)\), \(h'\mbk{F(x);d}=h(d)\), \(\mathcal{U}h\mbk{F(x)}=\mathcal{U}h(0)\) and \(\mathcal{V}h\mbk{F(x)}=\mathcal{V}h(0)\). 
Consequently,  \eqref{eq:pfyinu}-\eqref{eq:uvy} follow from the basic chain rule of subdifferential  and \cite[Theorem~1]{Liu2026vu}.
\end{proof}



The enlargement is required to lie in \(\partial h(0)\) and to recover \(\partial h\) on the lineality space.
\begin{assumption}\label{assu:1}
For \(h\in \mathcal{H}\), the set-valued mapping \(\bar{\partial}_\varepsilon h(x)\colon\mathbb{R}^n\times \mathbb{R}_{\geq 0}\rightrightarrows\mathbb{R}^n \) satisfies the following conditions:
(1) \(\bar{\partial}_\varepsilon h(x)\) is convex and contained in \(\partial h(0)\) for all \((x,\varepsilon)\in \mathbb{R}^n\times\mathbb{R}_{\geq 0}\); (2) \(\bar{\partial}_0 h(u)=\partial h(u)\) for all \(u\in \mathcal{U}h(0)\)\,.
\end{assumption}
We can easily verify that \(\partial_\varepsilon h(x)\) satisfies Assumption \ref{assu:1}. 
Under this assumption the nested inclusion and the one-sided limits pass to the enlargement. 
\begin{proposition}\label{prop:vuesub}
For \(h\in\mathcal{H}\) and \((x,\varepsilon)\in\mathbb{R}^n\times\mathbb{R}_{\geq 0}\) with \(\bar{\partial}_\varepsilon h(x)\neq\emptyset\), if \(\bar{\partial}_\varepsilon h(x)\) satisfies Assumption \ref{assu:1}, then  the following holds:
\begin{gather}
\mathcal{V}_\varepsilon h(x)\subset \mathcal{V}h(0), \ \mathcal{U}h(0)\subset \mathcal{U}_{\varepsilon}h(x),\notag
\\
\limsup_{\substack{x\to u\\ \varepsilon\searrow 0}  } \mathcal{V}_\varepsilon h (x)=\mathcal{V}h(u),\text{ for all }u\in \mathcal{U}h(0)\,.\label{eq:sub-con}
\\
\liminf_{\substack{x\to u\\ \varepsilon\searrow 0} }\mathcal{U}_{\varepsilon}h(x)=\mathcal{U}h(u)\text{  for all }u\in \mathcal{U}h(0)\,.\label{eq:scin}
\end{gather}
\end{proposition}
\begin{proof}
By Assumption \ref{assu:1}, \(\bar{\partial}_\varepsilon h(x)\subset \partial h(0)\) for all \((x,\varepsilon)\in \mathbb{R}^n\times\mathbb{R}_{\geq 0}\)\,. Then the definition of  \(\mathcal{V}_\varepsilon h (x)\)  immediately yields \(\mathcal{V}_\varepsilon h(x)\subset\mathcal{V}h(0)\). 
It follows from \eqref{eq:4} and the uniqueness of the orthogonal complement that \(\mathcal{U}h(0)\subset \mathcal{U}_{\varepsilon}h(x)\) whenever \(\bar{\partial}_\varepsilon h(x)\neq\emptyset\). 
From the second part of Assumption \ref{assu:1} and \eqref{eq:uv-eq} we get 
\(\mathcal{V}_0h(u)=\mathcal{V}h(u)=\mathcal{V} h(0),\ \mathcal{U}_0 h(u)=\mathcal{U}h(u)=\mathcal{U} h(0)\) and therefore 
\(\mathcal{V}_\varepsilon h(x)\subset\mathcal{V}h(u)\) and \(\mathcal{U}h(u)\subset \mathcal{U}_{\varepsilon}h(x)\). 
Then applying Proposition \ref{prop:basicSemiC} can easily show \eqref{eq:sub-con} and \eqref{eq:scin}.
\end{proof}
The transversality below is a nonseparability condition; it is what lets the limits pass from \(h\) to \(f\).
\begin{lemma}\label{prop:subspaces}
Let \(A:\mathbb{R}^m\to\mathbb{R}^n\) be linear and let \(U\subset\mathbb{R}^n\) be a subspace.
Then \(\operatorname{Im} A\) cannot be separated from \(U\) if and only if
\(\ker A^\top\cap U^\perp=\{0\}\).
\end{lemma}
\begin{proof}
One has \(\operatorname{Im} A=(\ker A^\top)^\perp\) and \(U=(U^\perp)^\perp\).
Two subspaces of a Euclidean space fail to be separable if and only if the intersection of their orthogonals is \(\{0\}\).
\end{proof}

\subsection{Transfer and approximate leaves}
The transversality used by both theorems below is
\begin{equation}\label{eq:Meps-trans}
\ker F'(\bar x)^\top\cap\mathcal{V}h\bigl(F(\bar x)\bigr)=\{0\},
\end{equation}
equivalently \(\operatorname{Im} F'(\bar x)+\mathcal{U}h(F(\bar x))=\mathbb{R}^m\).

\begin{theorem}[transfer of \(\varepsilon\)-\(\mathcal{VU}\) under composition]\label{thm:eps-transfer}
Let \(f=h\circ F\) with \(h\in\mathcal{H}\) and \(F:\mathbb{R}^n\to\mathbb{R}^m\) of class \(C^1\).
Assume \(\bar x\in\mathcal{R}_{h,F}\).
The limits below are taken over those \((x,\varepsilon)\) for which \(\bar{\partial}_\varepsilon h(F(x))\neq\emptyset\).
If
\begin{gather}\label{eq:iscUehf}
\liminf_{\substack{x\to\bar x\\\varepsilon\searrow 0}}\mathcal{U}_\varepsilon h(F(x))=\mathcal{U}h(F(\bar x)),\\
\label{eq:outlvfe}
\limsup_{\substack{x\to\bar x\\\varepsilon\searrow 0}}\mathcal{V}_\varepsilon h(F(x))=\mathcal{V}h(F(\bar x)),
\end{gather}
and \eqref{eq:Meps-trans} holds, then
\begin{equation}\label{eq:inscuex}
\liminf_{\substack{x\to\bar x\\\varepsilon\searrow 0}}\mathcal{U}_\varepsilon f(x)=\mathcal{U}f(\bar x)
\end{equation}
and
\begin{equation}\label{eq:prop94}
\limsup_{\substack{x\to\bar x\\\varepsilon\searrow 0}}\mathcal{V}_\varepsilon f(x)=\mathcal{V}f(\bar x).
\end{equation}
\end{theorem}
\begin{proof}
The argument below uses the limit identities \eqref{eq:iscUehf} and \eqref{eq:outlvfe} and the transversality \eqref{eq:Meps-trans}.
Write \(A(x):=F'(x)\), \(C(x,\varepsilon):=\mathcal{U}_\varepsilon h(F(x))\) and \(S(x,\varepsilon):=\mathcal{V}_\varepsilon h(F(x))\), and set \(A:=A(\bar x)\), \(C:=\mathcal{U}h(F(\bar x))\), \(S:=C^\perp\). Proposition~\ref{prop:epsVUchain} gives \(\mathcal{U}_\varepsilon f(x)=A(x)^{-1}C(x,\varepsilon)\) and \(\mathcal{V}_\varepsilon f(x)=A(x)^\top S(x,\varepsilon)\). The exact spaces at \(\bar x\) are \(A^{-1}C\) and \(A^\top S\) by \(\bar x\in\mathcal{R}_{h,F}\) and \cite[Theorem~1]{Liu2026vu}.
Condition \eqref{eq:Meps-trans} is \(\ker A^\top\cap S=\{0\}\), equivalently \(\operatorname{Im} A+C=\mathbb{R}^m\), so \(\operatorname{Im} A\) cannot be separated from \(C\) (Lemma~\ref{prop:subspaces}). The maps here vary with \(x\): \(A(x)\to A\) because \(F\) is \(C^1\), while \(C(x,\varepsilon)\) and \(S(x,\varepsilon)\) have the inner and outer limits \eqref{eq:iscUehf}--\eqref{eq:outlvfe} by hypothesis. The inner limit is the inverse-image case of \cite[Theorem~4.32(a)]{Rockafellar1998}. The data are the subspaces \(C(x,\varepsilon)\), which are convex, and the linear maps \(A(x)\). Continuous convergence \(A(x)\to A\) is the \(C^1\) assumption on \(F\). Nonseparability of \(\operatorname{Im} A\) from \(C\) is \(\ker A^\top\cap C^\perp=\{0\}\), which is \eqref{eq:Meps-trans}. Theorem~4.32(a) therefore yields \eqref{eq:inscuex}.
For the outer limit, the subspaces \(\mathcal{V}_\varepsilon f(x)\) are the images of the converging linear maps \(F'(x)^\top\), so \cite[Theorem~5.53(b)]{Rockafellar1998} applies once four conditions are checked. First, \eqref{eq:outlvfe} is the outer limit of the value sets. Second, the horizon outer limit equals the horizon cone: since the sets are subspaces, Lemma~\ref{lem:horizonout} gives \(\limsup^\infty\mathcal{V}_\varepsilon h(F(x))=\limsup\mathcal{V}_\varepsilon h(F(x))\), and a subspace is a cone whose recession cone is itself \cite[Theorem~3.6]{Rockafellar1998}, so the horizon identity follows from \eqref{eq:outlvfe}. Third, the total graphical convergence \(F'(x)^\top\xrightarrow{t}F'(\bar x)^\top\) \cite[p.~185]{Rockafellar1998}: if \(v_k\to v\) and \(F'(x_k)^\top v_k\to w\), then \(w=F'(\bar x)^\top v\), which is graphical convergence \(F'(x)^\top\xrightarrow{g}F'(\bar x)^\top\), and the convexity of the graphs gives total graphical convergence by the criterion of \cite[Proposition~4.25]{Rockafellar1998} cited on p.~185. Fourth, the horizon qualification \(\bigl(F'(\bar x)^\top\bigr)^{\infty\,-1}(0)\cap\mathcal{V}h(F(\bar x))=\{0\}\) is \(\ker F'(\bar x)^\top\cap\mathcal{V}h(F(\bar x))=\{0\}\) \cite[p.~159]{Rockafellar1998}, which is \eqref{eq:Meps-trans}. This yields \eqref{eq:prop94}.
\end{proof}
Condition \eqref{eq:Meps-trans} is the nonseparability hypothesis of \cite[Theorem~4.32(a)]{Rockafellar1998} and the horizon qualification of \cite[Theorem~5.53(b)]{Rockafellar1998}. Membership of \(\bar x\) in \(\mathcal{R}_{h,F}\) gives only the exact identities \(\mathcal{V}f(\bar x)=F'(\bar x)^\top\mathcal{V}h(F(\bar x))\) and \(\mathcal{U}f(\bar x)=F'(\bar x)^{-1}\bigl(\mathcal{U}h(F(\bar x))\bigr)\); without \eqref{eq:Meps-trans} the \(\varepsilon\)-limits need not pass from \(h\) to \(f\).

The same transversality controls the approximate leaf.
\begin{theorem}[approximate active manifolds]\label{thm:Meps}
Let \(f=h\circ F\) with \(h\in\mathcal{H}\) and \(F:\mathbb{R}^n\to\mathbb{R}^m\) of class \(C^1\).
Write
\[
\mathcal{M}_\varepsilon(x)
:=
F^{-1}\bigl(F(x)+\mathcal{U}_\varepsilon h(F(x))\bigr),
\qquad
\mathcal{M}
:=
F^{-1}\bigl(\mathcal{U}h(F(\bar x))\bigr).
\]
The sets \(\mathcal{M}_\varepsilon(x)\) are considered for those \((x,\varepsilon)\) with \(\bar{\partial}_\varepsilon h(F(x))\neq\emptyset\).
Assume
\begin{enumerate}[label=\textnormal{(\alph*)}]
\item\label{it:Meps-c} \(\mathcal{U}h(F(\bar x))\subset\mathcal{U}_\varepsilon h(F(x))\) for all \((x,\varepsilon)\) sufficiently close to \((\bar x,0)\);
\item\label{it:Meps-d} \eqref{eq:Meps-trans} holds.
\end{enumerate}
Then there is a neighbourhood \(W\) of \(\bar x\) such that
\begin{enumerate}[label=\textnormal{(\roman*)}]
\item\label{it:Meps-isc}
every \(y^*\in\mathcal{M}\cap W\) lies in \(\liminf_{\substack{x\to\bar x\\\varepsilon\searrow 0}}\mathcal{M}_\varepsilon(x)\).
\end{enumerate}
If in addition \(\bar x\in\mathcal{R}_{h,F}\) in the sense of Assumption~\ref{as:2}, and \(\dim\mathcal{U}_\varepsilon h(F(x))=\dim\mathcal{U}h(F(\bar x))\) for all \((x,\varepsilon)\) sufficiently close to \((\bar x,0)\), then
\begin{enumerate}[label=\textnormal{(\roman*)}]
\setcounter{enumi}{1}
\item\label{it:Meps-dim}
\[
\mathcal{M}_\varepsilon(x)\cap W
\;\longrightarrow\;
\mathcal{M}\cap W
\]
in the Painlev\'e--Kuratowski sense, and
\[
T_{\mathcal{M}_\varepsilon(x)}(x)=\mathcal{U}_\varepsilon f(x)
\qquad\text{for all }x\in W\text{ with }x\in\mathcal{M}_\varepsilon(x).
\]
\end{enumerate}
\end{theorem}
\begin{proof}
Write \(U:=\mathcal{U}h(F(\bar x))\) and \(V:=U^\perp=\mathcal{V}h(F(\bar x))\).
Hypothesis \ref{it:Meps-c} gives \(U\subset\mathcal{U}_\varepsilon h(F(x))\) for every \((x,\varepsilon)\) near \((\bar x,0)\).
For any continuous \(F\) one has
\(\liminf F^{-1}(D_{x,\varepsilon})\subset F^{-1}(\liminf D_{x,\varepsilon})\)
\cite[Equation~4(8), p.~129]{Rockafellar1998}.

The exact affine preimage
\[
\mathcal{M}_0(x):=F^{-1}\bigl(F(x)+U\bigr)
\]
is therefore contained in \(\mathcal{M}_\varepsilon(x)\), and \(\mathcal{M}_0(\bar x)=\mathcal{M}\).
Define \(G:\mathbb{R}^n\times\mathbb{R}^n\to V\) by
\[
G(y;x)=P_V\bigl(F(y)-F(x)\bigr).
\]
The set \(\mathcal{M}_0(x)\) is the zero fibre of \(G(\,\cdot\,;x)\).
If \(y^*\in\mathcal{M}\) then \(G(y^*;\bar x)=P_VF(y^*)=0\).
The derivative \(D_yG(\bar x;\bar x)=P_VF'(\bar x)\) is surjective by \eqref{eq:Meps-trans}.
Since \(F\) is \(C^1\), the same derivative remains surjective at every \(y^*\in\mathcal{M}\) sufficiently close to \(\bar x\).
The implicit-function theorem therefore supplies, for each \(y^*\in\mathcal{M}\cap W\), a \(C^1\) correction
\(x\mapsto y(x)=y^*+R(y^*)\,v(x)\)
along a right inverse \(R(y^*)\) of \(P_VF'(y^*)\), with \(G(y(x);x)=0\) and \(y(x)\to y^*\) as \(x\to\bar x\).
The right inverse depends on \(y^*\), so this is pointwise membership in the inner limit, not a uniform chart on \(\mathcal{M}\cap W\).
Thus \(y^*\in\liminf_{x\to\bar x}\mathcal{M}_0(x)\subset\liminf\mathcal{M}_\varepsilon(x)\), which is \ref{it:Meps-isc}.

Now assume the dimension identity. Combined with \(U\subset\mathcal{U}_\varepsilon h(F(x))\) one obtains
\(\mathcal{U}_\varepsilon h(F(x))=U\), hence \(\mathcal{M}_\varepsilon(x)=\mathcal{M}_0(x)\) near \((\bar x,0)\).
The inner inclusion \(\mathcal{M}\cap W\subset\liminf(\mathcal{M}_\varepsilon(x)\cap W)\) is the pointwise membership already proved for \ref{it:Meps-isc}, restricted to \(W\).
For the outer inclusion, take \(y_k\in\mathcal{M}_\varepsilon(x_k)\cap W\) with \(x_k\to\bar x\), \(\varepsilon_k\searrow 0\) and \(y_k\to y\). For all large \(k\) the dimension identity gives \(\mathcal{M}_\varepsilon(x_k)=\mathcal{M}_0(x_k)\), so \(G(y_k;x_k)=0\). Continuity of \(F\) yields \(G(y;\bar x)=0\), hence \(y\in\mathcal{M}\). Since \(W\) may be taken closed, \(y\in W\). Thus \(\limsup(\mathcal{M}_\varepsilon(x)\cap W)\subset\mathcal{M}\cap W\).
The two inclusions are the Painlev\'e--Kuratowski convergence \(\mathcal{M}_\varepsilon(x)\cap W\to\mathcal{M}\cap W\).
The tangent space is
\[
T_{\mathcal{M}_\varepsilon(x)}(x)
=\{w:F'(x)w\in U\}
=\{w:F'(x)w\in\mathcal{U}_\varepsilon h(F(x))\}
=\mathcal{U}_\varepsilon f(x),
\]
where the last identity is \eqref{eq:uvefcom}, which requires \(x\in\mathcal{R}_{h,F}\).
\end{proof}
The reverse inclusion \(\liminf\mathcal{M}_\varepsilon(x)\subset\mathcal{M}\) is not claimed: the nesting gives only the inner limit of \(\mathcal{U}_\varepsilon h(F(x))\), and the outer limit is not assumed.
The transversality \eqref{eq:Meps-trans} is used as the surjectivity of \(P_VF'(\bar x)\) in the implicit-function step. Without it the zero fibre of \(G\) need not be a manifold, and the inner limit of \(\mathcal{M}_\varepsilon\) is not claimed.

The next example shows that the \(\varepsilon\)-limit need not pass from \(h\) to \(f\) when the transversality fails.
\begin{example}[transfer and the approximate leaf]\label{eg:transfer}
Let \(h=\lvert\cdot\rvert\) on \(\mathbb{R}\) and \(F:\mathbb{R}^2\to\mathbb{R}\), \(F(x_1,x_2)=x_1^2\).
Then \(F'(0)=0\), so \(\ker F'(0)^\top=\mathbb{R}=\mathcal{V}h(0)\) and \eqref{eq:Meps-trans} fails.
Here \(f(x)=x_1^2\) is smooth, so \(\mathcal{U}f(0)=\mathbb{R}^2\).
By Example~\ref{eg:1}, \(\partial_\varepsilon h(t)\) has positive length for every \(t\) and every \(\varepsilon>0\), so \(\operatorname{aff}\partial_\varepsilon h(t)=\mathbb{R}\) and \(\mathcal{U}_\varepsilon h(t)=\{0\}\) at every \(t\).
Hence \(\mathcal{U}_\varepsilon f(x)=\ker F'(x)\), which equals \(\operatorname{span}\{e_2\}\) when \(x_1\neq 0\) and \(\mathbb{R}^2\) when \(x_1=0\).
Thus
\[
\liminf_{\substack{x\to 0\\\varepsilon\searrow 0}}\mathcal{U}_\varepsilon f(x)=\operatorname{span}\{e_2\}\neq\mathbb{R}^2=\mathcal{U}f(0).
\]
The \(\varepsilon\)-limit does not pass from \(h\) to \(f\).
Membership of the origin in \(\mathcal{R}_{h,F}\) still gives the exact chain-rule identities \(\mathcal{V}f(0)=F'(0)^\top\mathcal{V}h(0)=\{0\}\) and \(\mathcal{U}f(0)=F'(0)^{-1}\bigl(\mathcal{U}h(0)\bigr)=\mathbb{R}^2\), by \cite[Theorem~1]{Liu2026vu}.
\end{example}
The next example computes the spaces for two composites: the unbiased layer, where the transversality holds and the limits pass, and max-pooling, where the transversality fails but the limit still holds.
\begin{example}[unbiased layer and max-pooling]\label{eg:layer}
Both constructions are composites \(h\circ F\) with \(h\in\mathcal{H}\) polyhedral.
\begin{enumerate}[label=\textnormal{(\roman*)}]
\item \emph{Unbiased ReLU layer.}
Let \(A\in\mathbb{R}^{m\times n}\) have full row rank and set \(F(x)=Ax\), \(h(y)=\sum_{i=1}^m\max\{y_i,0\}\).
Then \(h\in\mathcal{H}\) by separability, \(\partial h(0)=[0,1]^m\), \(\mathcal{U}h(0)=\{0\}\) and \(\mathcal{V}h(0)=\mathbb{R}^m\).
Full row rank means \(\ker A^\top=\{0\}\), so \eqref{eq:Meps-trans} holds at the origin.
The exact spaces are \(\mathcal{U}f(0)=A^{-1}(\{0\})=\ker A\) and \(\mathcal{V}f(0)=\operatorname{Im} A^\top\).
By Example~\ref{eg:relu} each factor \(\max\{y_i,0\}\) has \(\mathcal{U}_\varepsilon=\{0\}\) for every value and every \(\varepsilon>0\).
The enlargement of a separable sum is the product of the enlargements, so the identity of \cite[Lemma~3.4]{Liu2026vu} gives \(\mathcal{U}_\varepsilon h(y)=\{0\}^m\) for all \(y\) and all \(\varepsilon>0\).
Hence \(\mathcal{U}_\varepsilon f(x)=A^{-1}(\{0\})=\ker A=\mathcal{U}f(0)\), and the limits of Theorem~\ref{thm:eps-transfer} hold.
The leaf is \(\mathcal{M}=A^{-1}(\mathcal{U}h(0))=\ker A\), and \(\mathcal{M}_\varepsilon(x)=A^{-1}(Ax+\{0\})=x+\ker A\).
Every point of \(\mathcal{M}\) lies in \(\mathcal{M}_\varepsilon(x)\) for \(x\in\mathcal{M}\), which is the inner semicontinuity of Theorem~\ref{thm:Meps}.
\item \emph{Max-pooling.}
Let \(h(y)=\max\{y_1,y_2\}=\sigma_{\Delta_2}(y)\) and \(F(t)=(t,t)\).
At \(F(t)=(t,t)\) the two entries tie, so \(\partial h(F(t))=\operatorname{conv}\{e_1,e_2\}\) and \(\mathcal{V}h(F(t))=\operatorname{span}\{e_1-e_2\}\).
Here \(F'(t)=(1,1)^\top\), so \(\ker F'(t)^\top=\operatorname{span}\{e_1-e_2\}=\mathcal{V}h(F(t))\) and \eqref{eq:Meps-trans} fails at every \(t\).
The composite \(f(t)=t\) is smooth at every \(t\), so \(\mathcal{U}f(t)=\mathbb{R}\) for all \(t\).
For every \(t\) and every \(\varepsilon>0\), \(\langle s,(t,t)\rangle=t\) for all \(s\in\Delta_2\), hence \(\partial_\varepsilon h(t,t)=\Delta_2\) and \(\mathcal{U}_\varepsilon h(t,t)=\operatorname{span}\{e_1+e_2\}\).
Thus \(\mathcal{U}_\varepsilon f(t)=F'(t)^{-1}(\operatorname{span}\{e_1+e_2\})=\mathbb{R}=\mathcal{U}f(t)\) at every \(t\): the limit holds even though transversality fails.
Failure of \eqref{eq:Meps-trans} is necessary for the counterexample of Example~\ref{eg:transfer}, not sufficient.
\end{enumerate}
\end{example}

\subsection{Extreme-point enlargement}
One enlargement that meets Assumption~\ref{assu:1} is the extreme-point construction below. The polyhedral enlargement is the next subsection.

For \(h\in\mathcal{H}\) such that \(\partial h(0)\) contains no lines, we consider the following object
\begin{equation}\label{eq:defdeins}
\bar{\partial}_\varepsilon h(x)\mathrel{\mathop:}=\overline{\co }\mcb{s\in\ext\partial h(0)\colon \mang{s,x}\geq h(x)-\varepsilon}.
\end{equation}
Here \(\operatorname{ext}\partial h(0)\) follows \cite[Theorem~18.5]{Rockafellar1970}: when \(\partial h(0)\) contains no lines it is the set of extreme points; extreme rays are included when they exist.
\begin{proposition}\label{th:1}
Consider \(h\in\mathcal{H}\)  such that \(\partial h(0)\) contains no lines. Let \(\bar{\partial}_\varepsilon h(x)\) be defined as in \eqref{eq:defdeins} for \(x\in\mathbb{R}^n\) and \(\varepsilon\ge 0\). Then
\begin{gather} \label{eq:2}
\partial h(x)=\bar\partial_0h(x)\subseteq \bar{\partial}_\varepsilon h(x)\subseteq \partial_\varepsilon h(x) \ 
\\
\bar{\partial}_\varepsilon h(x)\subset\partial h(0)=\bar{\partial}_\varepsilon h(0)\ \label{eq:1}
\end{gather}
If \(\partial h(x)\neq\emptyset\), the first inclusion forces \(\bar{\partial}_\varepsilon h(x)\neq\emptyset\) and
\begin{gather}
\mathcal{V}h(x)\subseteq \mathcal{V}_\varepsilon h(x)\subset \mathcal{V}h(0), \ \mathcal{U}h(0)\subset \mathcal{U}_{\varepsilon}h(x)\subseteq \mathcal{U}h(x),\  
\notag
\end{gather}
If additionally \(u\in \mathcal{U}h(0)\), then \begin{gather}
\bar{\partial}_\varepsilon h(x+u)=\bar{\partial}_\varepsilon h(x),\label{eq:deee}\\   \mathcal{V}_\varepsilon h(u)= \mathcal{V}h(u), \ \mathcal{U}_\varepsilon h(u)= \mathcal{U}h(u),\ \label{eq:vuequivu}
\end{gather}
\end{proposition}
\begin{proof}
As \(\partial h(0)\) contains no lines, \cite[Theorem 18.5]{Rockafellar1970} gives \(\partial h(0)=\co\ext\partial h(0)\). For \(s\in\partial h(x)\), write \(s=\sum\alpha_is_i\) with extreme points \(s_i\) and weights summing to one. If some \(s_j\) fails to expose \(x\), then \(\mang{s_j,x}<h(x)\), and the convex combination gives \(\mang{s,x}<h(x)\), a contradiction. Hence every extreme point in the combination exposes \(x\), so \(s\in\bar\partial_0h(x)\). The reverse inclusion is the definition. The sandwich \(\bar\partial_0h\subset\bar{\partial}_\varepsilon h\subset\partial_\varepsilon h\) is the definition together with Proposition~\ref{prop:sub}\ref{it:1-4}, and \(\bar{\partial}_\varepsilon h(x)\subset\partial h(0)\) is Proposition~\ref{prop:sub}\ref{it:1-5}. At \(x=0\), \(\bar{\partial}_\varepsilon h(0)=\partial h(0)\) by the same theorem. The nesting of the \(\mathcal{V}\) and \(\mathcal{U}\) spaces is the nesting of the sets in \eqref{eq:2} and \eqref{eq:1}. For \eqref{eq:deee}, the identity \(h(x+u)=h(x)+\mang{g_0,u}\) of \eqref{eq:hxu} and \(s-g_0\in\mathcal{V}h(0)\) give \(\mang{s,x+u}=\mang{s,x}+\mang{g_0,u}\), so the exposing inequality is unchanged and the enlargement is invariant. Setting \(x=0\) gives \eqref{eq:vuequivu}.
\end{proof} 
It is clear from Theorem \ref{th:1} that the \(\bar{\partial}_\varepsilon h(x)\) satisfies Assumption \ref{assu:1}\,. 
The next four examples compute the extreme-point enlargement for the functions of Examples~\ref{eg:1},~\ref{eg:relu} and~\ref{eg:subli}.
\begin{example}[Example~\ref{eg:1} continued]
	\[\bar{\partial}_\varepsilon h(x)=\begin{cases}
		\mcb{-1}&\text{ if }x<-\varepsilon/2,\\
		[-1,1] &\text{ if }-\varepsilon/2\leq x\leq \varepsilon/2,\\
		\{1\} &\text{ if }x>\varepsilon/2.
	\end{cases}\]
	From the expression of \(\bar{\partial}_\varepsilon h(x)\)  
	we can use an argument similar to that in Example \ref{eg:1} to show that  \(\liminf _{\substack{x\to 0\\\varepsilon\searrow 0}}\bar{\partial}_\varepsilon h(x)\not=\partial h(0)\). 
\end{example}
\begin{example}[Example~\ref{eg:relu} continued]\label{eg:relu-bar}
The set \(\partial h(0)=[0,1]\) has extreme points \(\{0,1\}\). Definition \eqref{eq:defdeins} therefore gives
\[
\bar{\partial}_\varepsilon h(t)
=
\overline{\operatorname{co}}\bigl\{s\in\{0,1\}:st\ge h(t)-\varepsilon\bigr\}
=
\begin{cases}
\{1\} & t>\varepsilon,\\
[0,1] & \lvert t\rvert\le\varepsilon,\\
\{0\} & t<-\varepsilon.
\end{cases}
\]
Hence \(\mathcal{U}_\varepsilon h(t)=\mathbb{R}\) on each open ray \(\lvert t\rvert>\varepsilon\), while \(\mathcal{U}_\varepsilon h(t)=\{0\}=\mathcal{U}h(0)\) throughout the strip \(\lvert t\rvert\le\varepsilon\). The extreme-point enlargement, unlike \(\partial_\varepsilon\), does identify the two linear pieces once \(\lvert t\rvert>\varepsilon\). Inner semicontinuity of \(\bar{\partial}_\varepsilon h\) as \((t,\varepsilon)\to(0,0)\) still fails, by the same sequences as in Example~\ref{eg:1}.
\end{example}
\begin{example}[Example~\ref{eg:subli} continued]
	For \(h(X)=\lambda_1(X)\) the extreme-point enlargement is \(\bar{\partial}_\varepsilon\lambda_1(X)=\co\{dd^T:\|d\|=1,\ d^TXd\geq\lambda_1(X)-\varepsilon\}\). The inner approximation of \cite{Oustry20001} is contained in it \cite[Section~6.2.3]{Liu2019Subdifferential}.
\end{example}
\begin{example}[Example~\ref{eg:subli} continued] 
	\begin{enumerate}[label=\textnormal{(\roman*)}]
		\item 
		For \(h(x)=\sigma_{\Delta^n}(x)\) the exact and \(\varepsilon\)-active index sets are
		\begin{equation}\label{eq:Ieps}
			I(x)=\mcb{i\in I\colon x_i=h(x)},\qquad
			I_{\varepsilon}(x)=\mcb{i\in I\colon x_i\geq h(x)-\varepsilon}.
		\end{equation} 
		Then
		\begin{align*}
		\partial_{\varepsilon}h(x)
		&=\mcb{s\in\Delta^n\colon\mang{s,x}\geq\max_{i\in I}x_i-\varepsilon}\\
		&=\mcb{s\in\Delta^n\colon \sum_{i\in I}s_i\mbk{x_i-\max_{i\in I}x_i+\varepsilon}\geq0},
		\end{align*}
		and \(\bar{\partial}_\varepsilon h(x)=\co \mcb{e_i\colon i\in I_\varepsilon(x)}\), where \(I_\varepsilon\) is defined in \eqref{eq:Ieps}. 
		\item {\(h(x)=\norm{x}=\sigma_{\mathbb{B}_1}(x)\)}
\begin{align*}
\partial_{\varepsilon}h(x)&=\mcb{s\in\mathbb{R}^n\colon\normo{s}\leq 1,\ \mang{s,x}\geq \norm{x}-\varepsilon}\,,\\
\bar{\partial}_\varepsilon h(x)&=\co  \mcb{s\in\mathbb{R}^n\colon s_i=\pm1,\ s_j=0,\ j\not=i,\ i=1,\cdots,n,}\\
&\hspace{2em}\cap\mcb{s\in\mathbb{R}^n\colon\mang{s,x}\geq \norm{x}-\varepsilon}\,,
\end{align*}
	\end{enumerate}
\end{example}
The extreme-point enlargement meets Assumption~\ref{assu:1}, so the two transfer theorems apply to it directly.
\begin{corollary}[extreme-point enlargement]\label{cor:ext-transfer}
Let \(h\in\mathcal{H}\) with \(\partial h(0)\) containing no lines, and let \(\bar\partial_\varepsilon h\) be the enlargement \eqref{eq:defdeins}.
Let \(F:\mathbb{R}^n\to\mathbb{R}^m\) be of class \(C^1\) and \(f=h\circ F\).
Assume
\begin{enumerate}[label=\textnormal{(\alph*)}]
\item \(F(\bar x)\in\mathcal{U}h(0)\);
\item \eqref{eq:Meps-trans} holds.
\end{enumerate}
Then, writing \(\mathcal{M}\) and \(\mathcal{M}_\varepsilon(x)\) as in Theorem~\ref{thm:Meps}, there is a neighbourhood \(W\) of \(\bar x\) such that
\[
\mathcal{M}\cap W\subset\liminf_{\substack{x\to\bar x\\\varepsilon\searrow 0}}\mathcal{M}_\varepsilon(x).
\]
If in addition \(\bar x\in\mathcal{R}_{h,F}\), then the limits of Theorem~\ref{thm:eps-transfer} hold:
\[
\liminf_{\substack{x\to\bar x\\\varepsilon\searrow 0}}\mathcal{U}_\varepsilon f(x)=\mathcal{U}f(\bar x),\qquad
\limsup_{\substack{x\to\bar x\\\varepsilon\searrow 0}}\mathcal{V}_\varepsilon f(x)=\mathcal{V}f(\bar x).
\]
If moreover \(\dim\mathcal{U}_\varepsilon h(F(x))=\dim\mathcal{U}h(F(\bar x))\) near \((\bar x,0)\), then
\[
\mathcal{M}_\varepsilon(x)\cap W\to\mathcal{M}\cap W,\qquad T_{\mathcal{M}_\varepsilon(x)}(x)=\mathcal{U}_\varepsilon f(x).
\]
\end{corollary}
\begin{proof}
Proposition~\ref{th:1} gives the nesting \(\mathcal{U}h(0)\subset\mathcal{U}_\varepsilon h(x)\). Since \(F(\bar x)\in\mathcal{U}h(0)\), this is hypothesis \ref{it:Meps-c} of Theorem~\ref{thm:Meps}. Together with \eqref{eq:Meps-trans}, that theorem's first conclusion gives the inner semicontinuity. The same nesting and Proposition~\ref{prop:vuesub} supply the limit identities of Theorem~\ref{thm:eps-transfer}: since \(F(x)\to F(\bar x)\in\mathcal{U}h(0)\), that proposition applies at \(F(\bar x)\). The theorem requires \(\bar x\in\mathcal{R}_{h,F}\) and gives the two limits. The dimension identity and \(\bar x\in\mathcal{R}_{h,F}\) give the Painlev\'e--Kuratowski convergence and the tangent identity by the second conclusion of Theorem~\ref{thm:Meps}.
\end{proof}
The next example records that the unbiased layer meets the hypotheses of the corollary.
\begin{example}[Example~\ref{eg:layer} continued]\label{eg:layer-ext}
For the unbiased layer of Example~\ref{eg:layer}, \(\partial h(0)=[0,1]^m\) contains no lines, so the extreme-point enlargement \eqref{eq:defdeins} is available. Full row rank gives \eqref{eq:Meps-trans} at the origin, and \(F(0)=0\in\mathcal{U}h(0)\). The computation of that example gives \(\mathcal{U}_\varepsilon f(x)=\ker A=\mathcal{U}f(0)\) for all \(x\) and all \(\varepsilon>0\), so both limits of Corollary~\ref{cor:ext-transfer} hold as equalities, and the leaf identity \(\mathcal{M}_\varepsilon(x)=x+\ker A\) is the inner semicontinuity.
\end{example}

\subsection{Polyhedral enlargement}
The extreme-point enlargement requires \(\partial h(0)\) to contain no lines. For polyhedral \(h\in\mathcal{H}\) the enlargement is built from the active sets of the representation \eqref{eq:reprPolyh}.
To construct the subdifferential enlargement we first consider the $\varepsilon$-activity index set:
\begin{equation}\label{ind:epspoly}
I_\varepsilon(x)=\mcb{i\in I\colon \mang{a^i,x}\geq h(x)-\varepsilon}\,. 
\end{equation}
In view of the expression of \(\partial h(x)\) in \eqref{eq:polyex} we define the associated enlargement as follows 
\begin{equation}\label{eq:delpoly}
\bar{\partial}_\varepsilon h(x)\mathrel{\mathop:}=\operatorname{co}\{a^i:i\in I_\varepsilon(x)\}+\operatorname{pos}\{c^j:j\in J(x)\},
\end{equation}
with the convention \(\operatorname{co}\emptyset=\{0\}\). If \(I_\varepsilon(x)\neq\emptyset\) this is
\[
\left\{\sum_{i\in I_\varepsilon(x)}\alpha_ia^i+\sum_{j\in J(x)}\beta_jc^j
\colon\sum_{i\in I_\varepsilon(x)}\alpha_i =1,\ 
 \alpha_i\geq 0,\  \beta_j\geq0
\right\}.
\]
The enlargement meets Assumption~\ref{assu:1}.
\begin{theorem}[polyhedral enlargement]\label{thm:poly-enlarg}
Consider a polyhedral function \(h\in \mathcal{H}\) represented in the form of \eqref{eq:reprPolyh}\,. 
\begin{enumerate}[label=\textnormal{(\roman*)}]
\item Its subdifferential enlargement defined in \eqref{eq:delpoly}  satisfies 
\begin{equation}\label{eq:delepspo}
\partial h(x)\subset\bar{\partial}_\varepsilon h(x)\subset\partial_\varepsilon h(x)\,,
\end{equation} 
and therefore it satisfies Assumption \ref{assu:1}\,.
\item for any \(\bar{x}\in\dom h\), when \(x\) is sufficiently close to \(\bar{x}\) and \(\varepsilon\) is small enough we have \(\bar{\partial}_\varepsilon h(x)\subset \partial h(\bar{x})\) 
\item The subspace \(\mathcal{V}_\varepsilon(x)\)  associated with \(\bar{\partial}_\varepsilon h(x)\) can be explicitly 
represented as follows. 
\begin{gather}
\mathcal{V}_\varepsilon h(x)=\mcb{\sum_{i\in I_\varepsilon(x)}\alpha_ia^i+
\sum_{j\in J(x)}\beta_jc^j\colon\sum_{i\in I_\varepsilon(x)}\alpha_i=0,\;
\beta_j\in\RR}\label{eq:polyveps}\,.
\end{gather}\end{enumerate}
\end{theorem}
\begin{proof}
Obviously \(I(x)\subseteq I_{\varepsilon}(x)\)\,. 
From the expressions of \(\partial  h(x)\) and \(\bar{\partial}_\varepsilon h(x)\) in \eqref{eq:polyex} and \eqref{eq:delpoly} we immediately get \(\partial  h(x)\subseteq\bar{\partial}_\varepsilon h(x)\)\,. 
To see the second half of \eqref{eq:delepspo}, notice that
\[
\partial h(0)=\left\{\sum_{i\in I}\alpha_ia^i+\sum_{j\in J}\beta_jc^j
\colon\sum_{i\in I}\alpha_i =1,\ 
 \alpha_i\geq 0,\  \beta_j\geq0
\right\}
\]
and \(\sum_{i\in I_\varepsilon(x)}\alpha_ia^i+\sum_{j\in J(x)}\beta_jc^j\in \partial h(0)\)\,.  
We can write \(\partial_\varepsilon h(x)=\mcb{s\in\partial h(0)\colon \mang{s,x}\geq h(x)-\varepsilon}\) (see Proposition \ref{prop:sub}\ref{it:1-4}). Then from the definition of \(I_\varepsilon(x)\), \(J(x)\) and \(\bar{\partial}_\varepsilon h(x)\) we can see 
\begin{gather}
 \mang{\sum_{i\in I_\varepsilon(x)}\alpha_ia^i+\sum_{j\in J(x)}\beta_jc^j,x}=\sum_{i\in I_\varepsilon(x)}\alpha_i\mang{a^i,x}\\\geq\sum_{i\in I_\varepsilon(x)}\alpha_i (h(x)-\varepsilon)=h(x)-\varepsilon.
 \end{gather}
 Therefore \(\sum_{i\in I_\varepsilon(x)}\alpha_ia^i+\sum_{j\in J(x)}\beta_jc^j\in \partial_\varepsilon h(x)\) and \(\bar{\partial}_\varepsilon h(x)\subset\partial_\varepsilon h(x)\)\,. 

Consider \(h\) as a max-function of the form \(f (x) = \max_{i\in I} f_i(x) +\delta_C\)  where \(f_i(x)=\mang{a^i,x}\) and \(C=\mcb{y\in \mathbb{R}^n\colon\mang{c^j,x}\leq 0,j\in J}\). Fix \(\bar x\in C\). For each \(i\notin I(\bar x)\) the slack \(h(\bar x)-\mang{a^i,\bar x}\) is strictly positive. By continuity of the finitely many linear forms, there exist \(\bar\varepsilon>0\) and a neighbourhood of \(\bar x\) on which, for every such \(i\) and every \(\varepsilon\in[0,\bar\varepsilon]\), \(\mang{a^i,x}<h(x)-\varepsilon\), so \(i\notin I_\varepsilon(x)\). Hence \(I_\varepsilon(x)\subset I(\bar x)\) for all \(x\) in that neighbourhood and all \(\varepsilon\in[0,\bar\varepsilon]\). It follows from the expressions of \(\partial h(x)\) and \(\bar{\partial}_\varepsilon h(x)\) in \eqref{eq:polyex} and \eqref{eq:delpoly} respectively that \(\bar{\partial}_\varepsilon h(x)\subset \partial h(\bar{x})\).  

 We can easily verify that \( \bar{\partial}_\varepsilon h(x) \) is convex. 
 From Definition \ref{def:epsVU} we have \(\mathcal{V}_\varepsilon(x)=\spn (\bar{\partial}_\varepsilon h(x)-g)\) for an arbitrary \(g\in \bar{\partial}_\varepsilon h(x)
\)\,.  Then simple calculations just give 
 \eqref{eq:polyveps}\,.
\end{proof}

The next example computes the polyhedral enlargement for \(h=\max\{x_1,x_2\}\).
\begin{example}[polyhedral enlargement]\label{eg:delpoly}
Take \(h(x)=\max\{x_1,x_2\}\) on \(\mathbb{R}^2\). Then \(I=\{1,2\}\), \(J=\emptyset\), \(a^1=(1,0)\), \(a^2=(0,1)\), and
\[
\partial h(0)=\operatorname{co}\{(1,0),(0,1)\}.
\]
At \(x=(1,0)\) one has \(h(x)=1\) and \(I(x)=\{1\}\), so \(\partial h(x)=\{(1,0)\}\), \(\mathcal{V}h(x)=\{0\}\) and \(\mathcal{U}h(x)=\mathbb{R}^2\). For \(\varepsilon\in(0,1)\),
\[
I_\varepsilon(x)=\{1\},\qquad
\bar\partial_\varepsilon h(x)=\{(1,0)\},
\]
while for \(\varepsilon\ge 1\) one has \(I_\varepsilon(x)=\{1,2\}\) and
\[
\bar\partial_\varepsilon h(x)=\operatorname{co}\{(1,0),(0,1)\}=\partial h(0).
\]
Thus \(\mathcal{V}_\varepsilon h(x)\) jumps from \(\{0\}\) to \(\operatorname{span}\{(1,-1)\}\) when \(\varepsilon\) crosses \(1\). At the diagonal point \(z=(1,1)\), \(I(z)=\{1,2\}\), \(\mathcal{V}h(z)=\operatorname{span}\{(1,-1)\}\) and \(\mathcal{U}h(z)=\operatorname{span}\{(1,1)\}\).
\end{example}
Both points of this example lie in their \(\mathcal{U}\)-space, as required by Proposition~\ref{prop:xinU}: \((1,0)\in\mathcal{U}h(1,0)\) and \((1,1)\in\mathcal{U}h(1,1)\).
%
%

\begin{proposition}\label{prop:polyepsvusub}
For a polyhedral function \(h\in\mathcal{H}\) and \(\bar x\in\dom h\), if the subspace  \eqref{def:vep} is defined based on \eqref{eq:delpoly} then 
\begin{equation} \notag
 \begin{array}{ccc}
  \limsup_{\substack{x\to \bar{x}\\ \varepsilon\searrow 0}}\mathcal{V}_\varepsilon h (x)=\mathcal{V}h(\bar x)
&\mbox{and}&
  \liminf_{\substack{x\to \bar{x}\\ \varepsilon\searrow 0}}\mathcal{U}_\varepsilon h(x)=\mathcal{U}h(\bar x)\,.
\end{array}
 \end{equation}
\end{proposition}
\begin{proof}
By the activity inclusion proved above for the polyhedral enlargement, for any \(\bar x\in C\) the
relaxed activity set satisfies \(I_\varepsilon (x)\subset I(\bar{x})\) when \(x\) is sufficiently close to \(\bar{x}\) and \(\varepsilon\) is small enough, where the inclusion holds as an equality if \(\bar{x}\in\func{int}C\). 

  In view of the expressions \eqref{eq:delpoly} and \eqref{eq:polyex}, the inclusion \(\bar{\partial}_\varepsilon f(x)\subset \partial f(\bar{x})\) always holds, and it becomes an identity if \(\bar{x}\in\func{int}C\).  
 Hence, by \eqref{eq:polyveps},
 for all $ \bar x\in C$
when \((x,\varepsilon)\) are sufficiently close to \((\bar{x},0)\),
\[ \begin{array}{ccc}
\mathcal{V}_\varepsilon h (x)\subset\mathcal{V}h(\bar x)&\mbox{and}&
  \mathcal{U}_\varepsilon h(x)\supset\mathcal{U}h(\bar x)\,,
  \end{array}\]
  And the conclusion holds by Proposition \ref{prop:basicSemiC}.
\end{proof}
The exact \(\mathcal{VU}\) spaces of a polyhedral \(h\in\mathcal{H}\) are explicit, and the one-sided limits hold at every point of \(\dom h\).
\begin{proposition}\label{prop:polyVU}
Let \(h\in\mathcal{H}\) be polyhedral and \(\bar x\in\dom h\). Then
\begin{gather}
\mathcal{V}h(\bar x)=\mcb{\sum_{i\in I(\bar x)}\alpha_ia^i+
\sum_{j\in J(\bar x)}\beta_jc^j\colon\sum_{i\in I(\bar x)}\alpha_i=0,\;
\beta_j\in\RR}\,,\label{eq:vx}
\end{gather}
and
\begin{equation}\label{eq:polysemico}
\limsup_{x\rightarrow \bar{x}} \mathcal{V}h(x)=\mathcal{V}h(\bar x),\ \liminf_{x\rightarrow \bar{x}}\mathcal{U}h(x)=\mathcal{U}h(\bar x)\,.
\end{equation}
\end{proposition}
\begin{proof}
The formula \eqref{eq:vx} is \eqref{eq:polyex} and \eqref{eq:defVU}. From the continuity of the linear terms, \(I(x)\subset I(\bar x)\) and \(J(x)\subset J(\bar x)\) for \(x\) near \(\bar x\), so Proposition~\ref{prop:basicSemiC} gives \eqref{eq:polysemico}.
\end{proof}
The next example checks the one-sided limits for the rectifier.
\begin{example}[Example~\ref{eg:relu} continued]\label{eg:relu-poly}
Let \(h(t)=\max\{t,0\}\). Then \(I=\{1\}\), \(a^1=1\), and \(I(t)=\{1\}\) for \(t\ge 0\), while \(I(t)=\emptyset\) for \(t<0\).
For \(t>\varepsilon>0\), \(I_\varepsilon(t)=\{1\}\), so the polyhedral enlargement \eqref{eq:delpoly} gives \(\bar\partial_\varepsilon h(t)=\{1\}\), hence \(\mathcal{V}_\varepsilon h(t)=\{0\}=\mathcal{V}h(t)\) and \(\mathcal{U}_\varepsilon h(t)=\mathbb{R}=\mathcal{U}h(t)\).
As \((t,\varepsilon)\to(t_0,0)\) with \(t_0>0\), the activity stays \(\{1\}\) for small \(\varepsilon\), and both limits of Proposition~\ref{prop:polyepsvusub} hold as equalities.
At \(t_0=0\) the same proposition gives the one-sided limits, while the two-sided limit in \(\varepsilon\) fails by Remark~\ref{rem:relu-miss}.
\end{example}
The polyhedral enlargement meets Assumption~\ref{assu:1}, so the transfer theorem applies to it.
\begin{corollary}[polyhedral outer function]\label{cor:poly-transfer}
Let \(h\in\mathcal{H}\) be polyhedral and let \(\bar\partial_\varepsilon h\) be the enlargement \eqref{eq:delpoly}.
Let \(F:\mathbb{R}^n\to\mathbb{R}^m\) be of class \(C^1\), \(f=h\circ F\), \(\bar x\in\mathcal{R}_{h,F}\), and assume \eqref{eq:Meps-trans}.
Then the limits of Theorem~\ref{thm:eps-transfer} hold.
\end{corollary}
\begin{proof}
Theorem~\ref{thm:poly-enlarg} gives that the enlargement satisfies Assumption~\ref{assu:1}, and Proposition~\ref{prop:polyepsvusub} supplies the one-sided limits at every point of \(\dom h\). Since \(\bar x\in\mathcal{R}_{h,F}\) gives \(F(\bar x)\in\dom h\), the limits hold as \(F(x)\to F(\bar x)\). Theorem~\ref{thm:eps-transfer} therefore gives the conclusion.
\end{proof}
The next example records that the same layer meets the hypotheses of the polyhedral corollary.
\begin{example}[Example~\ref{eg:layer} continued]\label{eg:layer-poly}
The unbiased layer of Example~\ref{eg:layer} is polyhedral, so the enlargement \eqref{eq:delpoly} applies. The same computation gives \(\mathcal{U}_\varepsilon f(x)=\ker A=\mathcal{U}f(0)\) for all \(x\) and all \(\varepsilon>0\), which is the conclusion of Corollary~\ref{cor:poly-transfer}. The rectifier of Example~\ref{eg:relu-poly} is the one-dimensional case with \(F\) the identity.
\end{example}

\section{Partial smoothness}
This section records partial smoothness of \(h\in\mathcal{H}\) relative to the lineality space, and the composite statement under the same transversality as Theorems~\ref{thm:eps-transfer} and~\ref{thm:Meps}. It is not used by those theorems. Throughout, the manifold is of class \(C^2\) and \(h\) is proper, closed and sublinear. Tilt stability is left aside. The identification of the active manifold under prox-regularity is~\cite{Hare2007identify}; the constant-rank view of partial smoothness is~\cite{Lewis2018rank}. 

A \(C^k\)-manifold and its normal space are as in \cite[Definition~2.1]{Lewis2002active}. A proper closed convex function \(f\) is \emph{partly smooth} at \(\bar x\) relative to a \(C^2\)-manifold \(\mathcal{M}\) if \(\partial f(\bar x)\neq\emptyset\), \(f\) restricted to \(\mathcal{M}\) is \(C^2\), \(N_{\mathcal{M}}(\bar x)=\mathcal{V}f(\bar x)\), and \(\partial f\) is continuous at \(\bar x\) relative to \(\mathcal{M}\) \cite[Definition~2.7]{Lewis2002active}. It is partly smooth relative to \(\mathcal{M}\) if this holds at every point of \(\mathcal{M}\).

The nuclear norm is partly smooth relative to the constant-rank manifolds \cite[Table~1]{doi:10.1137/16M106340X}. The identification \(\mathcal{U}_\varepsilon\|\cdot\|_*(0)=\mathcal{U}\|\cdot\|_*(0)=\{0\}\) is Proposition~\ref{prop:sub}\ref{it:1-5} together with Proposition~\ref{prop:U0int}, since the spectral-norm ball has nonempty interior.
\begin{proposition}[semicontinuity along a partly smooth manifold]\label{prop:PSlim}
If a function \(g\colon\mathbb{R}^n\rightarrow\overbar{\mathbb{R}}\) is partly smooth relative to the set \(\mathcal{M}\subset\mathbb{R}^n\) then for all \(\bar{x}\in\mathcal{M}\) we have 
\begin{gather}
\limsup_{x \urightarrow[\mathcal{M}]{} \bar{x}}\mathcal{V}g(x)=\mathcal{V}g(\bar x)\label{eq:prop51}\\
\liminf_{x \urightarrow[\mathcal{M}]{} \bar{x}}\mathcal{U}g(x)=\mathcal{U}g(\bar x).\label{eq:prop52}
\end{gather}
\end{proposition}
\begin{proof}
As \(g\) is partly smooth at any point \(x\in\mathcal{M}\) relative to \(\mathcal{M}\), we have \(N_{\mathcal{M}}(x)=\mathcal{V}g(x)\) holds for all \(x\in\mathcal{M}\). Consequently, \eqref{eq:prop51} follows from the outer semicontinuity of normal cone (see \cite[Proposition 6.6]{Rockafellar1998}). 
The property \(N_{\mathcal{M}}(x)=\mathcal{V}g(x)\) is equivalent to  \(T_{\mathcal{M}}(x)=\mathcal{U}g(x)\). 
To show \eqref{eq:prop52}, it suffices to show that the mapping \(T_{\mathcal{M}}\) is inner semicontinuous relative to \(\mathcal{M}\). \cite[Example~6.8]{Rockafellar1998} gives Clarke regularity of a smooth manifold, and \cite[Corollary~6.29(g)]{Rockafellar1998} then gives the inner semicontinuity of \(T_{\mathcal{M}}\). 
\end{proof}
The next example shows the semicontinuity along an activation manifold, and the jump on leaving it.
\begin{example}[Example~\ref{eg:layer} continued]\label{eg:act-mfld}
Fix a sign pattern for the unbiased layer of Example~\ref{eg:layer} and let \(\mathcal{M}\) be the set of inputs realizing it.
On the relative interior every active coordinate of \(Ax\) is strictly positive and every inactive one strictly negative, so \(\partial h(F(x))\) is a singleton, \(\mathcal{V}h(F(x))=\{0\}\) and \(\mathcal{U}h(F(x))=\mathbb{R}^m\), both constant.
Hence \(\mathcal{V}f\) and \(\mathcal{U}f\) are constant along \(\mathcal{M}\), and the two limits of Proposition~\ref{prop:PSlim} hold as equalities.
When an input approaches the boundary of \(\mathcal{M}\), one coordinate of \(Ax\) hits zero, \(\partial h\) jumps from a singleton to an interval, and \(\mathcal{V}f\) jumps up. The semicontinuity is along \(\mathcal{M}\), not off it.
\end{example}
For \(h\in\mathcal{H}\), it is shown in \cite[Proposition 5.3]{hare2006functions} that \(h\) is partly smooth at 0 relative to \(\mathcal{U}h(0)\), which is also the lineality space of \(h\). 
With the help of Proposition \ref{prop:2}, we can have the following extension of this result. 
\begin{proposition}[partial smoothness relative to the lineality space]\label{pron:3}
If \(h\in\mathcal{H}\) then \(h\) is partly smooth relative to \(\mathcal{U}h(0)\).
\end{proposition}
\begin{proof}
At \(0\) this is \cite[Proposition~5.3]{hare2006functions}. Along the lineality space, Proposition~\ref{prop:2} gives translation invariance of the subspaces and of the subdifferential, which is the normal-space identity. The \(C^2\) restriction is the same identity: for \(u\in\mathcal{U}h(0)\) one has \(h(x+u)=h(x)+\mang{g_0,u}\), so \(h\) restricted to \(x+\mathcal{U}h(0)\) is affine.
\end{proof}
The next example has a nontrivial lineality space.
\begin{example}[lineality of the max function]\label{eg:max-lin}
Let \(h(y)=\max_{1\le i\le m}y_i\). Then \(\partial h(0)=\Delta_m\), which lies in the hyperplane \(\sum s_i=1\) and has empty interior in \(\mathbb{R}^m\).
Hence \(\mathcal{V}h(0)=\{v:\sum v_i=0\}\) and \(\mathcal{U}h(0)=\operatorname{span}\{\mathbf{1}\}\).
Along that line, \(h(y+c\mathbf{1})=h(y)+c\), so \(h\) is affine and Proposition~\ref{pron:3} applies with a one-dimensional manifold.
For the rectifier and the \(\ell_1\) norm the subdifferential at the origin has nonempty interior, so \(\mathcal{U}h(0)=\{0\}\) and the manifold is a point.
\end{example}
%

\begin{proposition}[partial smoothness of the composite]\label{prop:PScomp}
Suppose \(f(x)=h\mbk{F(x)}\) for a smooth mapping \(F\colon\mathbb{R}^n\to\mathbb{R}^m\) and a function \(h\in\mathcal{H}\). 
If \(F\) is transversal to \(\mathcal{U}h(0)\) at some \(\bar{x}\in F\inv\mbk{\mathcal{U}h(0)}\), i.e.
\begin{equation}\label{eq:transver}
\Ker\mbk{F'(\bar{x})^\top}\cap N_{\mathcal{U}h(0)}\mbk{F(\bar{x})}=\mcb{0},
\end{equation}
then 
\begin{enumerate}[label=\textnormal{(\roman*)}]
\item \(f\) is partly smooth at \(\bar{x}\) relative to \( F\inv\mbk{\mathcal{U}h(0)}\);
\item if in addition a neighbourhood of \(\bar x\) in \(F^{-1}\bigl(\mathcal{U}h(0)\bigr)\) lies in \(\mathcal{R}_{h,F}\), then
\[
\mathcal{V}f(x)=N_{ F\inv\mbk{\mathcal{U}h(0)}}(x)
\]
for all \(x\in  F\inv\mbk{\mathcal{U}h(0)}\) sufficiently close to \(\bar{x}\);\label{it2:2}
\item under the same restriction to \(\mathcal{R}_{h,F}\), \(\partial f\) is continuous relative to \( F\inv\mbk{\mathcal{U}h(0)}\) and therefore \(f\) is partly smooth at  all \(x\in  F\inv\mbk{\mathcal{U}h(0)}\) sufficiently close to \(\bar{x}\) relative to \( F\inv\mbk{\mathcal{U}h(0)}\);
\end{enumerate}
\end{proposition}
\begin{proof}
\indent (i) This follows from Proposition \ref{pron:3} and the chain rule of partial smoothness in \cite[Theorem 4.2]{Lewis2002active}.

\indent (ii) By the given transversality, we have
\[
N_{ F\inv\mbk{\mathcal{U}h(0)}}(x)=F'(x)^\top N_{\mathcal{U}h(0)}(F(x))
\]
for all \(x\in  F\inv\mbk{\mathcal{U}h(0)}\) sufficiently close to \(\bar{x}\). 
From the partial smoothness of \(h\) relative to \(\mathcal{U}h(0)\) (see Proposition \ref{pron:3}) we have \(N_{\mathcal{U}h(0)}(F(x))=\mathcal{V}h\mbk{F(x)}\). 
Consequently, \(N_{ F\inv\mbk{\mathcal{U}h(0)}}(x)=F'(x)^\top\mathcal{V}h\mbk{F(x)}\). The last space equals \(\mathcal{V}f(x)\) by \cite[Theorem~1(i)]{Liu2026vu} once \(x\in\mathcal{R}_{h,F}\). If \(h\) is finite-valued this is automatic; otherwise it is the additional restriction stated in (ii).

\indent (iii) On that same piece of the leaf, \eqref{eq:pfyinu} holds. 
Hence the continuity of \(\partial f\) follows immediately from the continuity of \(F'(x)\). 
The restriction \(f\vert_{ F\inv\mbk{\mathcal{U}h(0)}}\) is smooth because of \eqref{eq:fsmoF}.
The chain rule on \(\mathcal{R}_{h,F}\) gives regularity of \(f\) along that piece. 
In view of \ref{it2:2} and the continuity of \(\partial f\), we see that \(f\) is partly smooth at  all \(x\in  F\inv\mbk{\mathcal{U}h(0)}\) sufficiently close to \(\bar{x}\) relative to \( F\inv\mbk{\mathcal{U}h(0)}\).
\end{proof}
The next example shows that the leaf of the unbiased layer is the manifold of Proposition~\ref{prop:PScomp}.
\begin{example}[Example~\ref{eg:layer} continued]\label{eg:layer-leaf}
For the unbiased layer of Example~\ref{eg:layer}, the leaf is \(\mathcal{M}=\ker A\).
Full row rank is the transversality \eqref{eq:transver}, so Proposition~\ref{prop:PScomp} gives that \(f=h\circ A\) is partly smooth at the origin relative to \(\ker A\).
On that leaf the sign pattern is fixed, as already noted in Example~\ref{eg:layer}.
\end{example}
The nondegeneracy condition for a composite \(g\circ F\) with \(F(\bar x)=0\) is the same transversality as \eqref{eq:transver}: the tangent space to \(\mathcal{U}g(0)\) at the origin is the normal cone to \(\partial g(0)\) at a relative-interior point.

\subsection{Polyhedral partial smoothness}
In Propositions \ref{prop: polyheH} and \ref{prop:polyheiHp} we have seen the special properties of polyhedral function in \(\mathcal{H}\). 
Moreover, the ``active" index sets in the proof of Proposition  \ref{prop:polyheiHp} define a smooth manifold
\[\mathcal{M}_{\bar{x}}\mathrel{\mathop:}=\mcb{x\in\mathbb{R}^n\colon I(x)=I\mbk{\bar{x}}\ \mathrm{ and }\ J(x)=J\mbk{\bar{x}}}\]
relative to which \(h\) is partly smooth (see \cite[Example 3.4]{Lewis2002active}).
Although the active manifold in the definition of partial smoothness is unique \cite[Corollary 4.2]{Hare2004}, the following result shows that \(\mathcal{M}_{\bar{x}}\) coincides locally with the affine set \(\bar{x}+\mathcal{U}h(\bar x)\). 
\begin{proposition}[Proposition 5.2 in \cite{Liu2019Subdifferential}]\label{prop:UeqM}
For polyhedral functions, the associated manifold of partial smoothness satisfies \(\mathcal{M}_{\bar{x}}\subset\bar{x}+\mathcal{U}h(\bar x)\). Moreover, there exists \(\epsilon>0\) such that \(\mathcal{M}_{\bar{x}}\cap B\mbk{\bar{x},\epsilon}=\bar{x}+\mbk{\mathcal{U}h(\bar x)\cap B\mbk{0,\epsilon}}\).
\end{proposition}
Here we have the following special case of Proposition \ref{prop:UeqM}. 
\begin{proposition}\label{prop:polyHUeqM}
For polyhedral \(h\in\mathcal{H}\) and \(\bar x\in\dom h\), Proposition~\ref{prop:UeqM} gives
\[
\mathcal{M}_{\bar x}\cap B(\bar x,\epsilon)
=
\mathcal{U}h(\bar x)\cap B(\bar x,\epsilon)
\]
for some \(\epsilon>0\). The ball is centred at \(\bar x\), not at the origin.
\end{proposition}
\begin{proof}
Proposition~\ref{prop:xinU} gives \(\bar x\in\mathcal{U}h(\bar x)\), so \(\bar x+\mathcal{U}h(\bar x)=\mathcal{U}h(\bar x)\). Proposition~\ref{prop:UeqM} then gives the displayed local identity. The ball is centred at \(\bar x\).
\end{proof}
The next example identifies the manifold for the rectifier.
\begin{example}[Example~\ref{eg:relu} continued]\label{eg:relu-mfld}
Let \(h(t)=\max\{t,0\}\) and \(\bar t=1\). Then \(I(1)=\{1\}\) and \(\mathcal{U}h(1)=\mathbb{R}\).
The active manifold is \(\mathcal{M}_1=\{t:I(t)=\{1\}\}=\{t>0\}\).
For \(\epsilon<1\), \(\mathcal{M}_1\cap B(1,\epsilon)=(1-\epsilon,1+\epsilon)=\mathcal{U}h(1)\cap B(1,\epsilon)\), which is the local identity of Proposition~\ref{prop:polyHUeqM}.
At the origin both branches achieve the value, so \(I(0)=\{1\}\) and \(\mathcal{U}h(0)=\{0\}\); the manifold is a point and the identity is trivial.
\end{example}

Along the leaf the activity freezes for small \(\varepsilon\).

\begin{proposition}[polyhedral outer function, along the leaf]\label{prop:Meps-poly}
Let \(f=h\circ F\) with \(F:\mathbb{R}^n\to\mathbb{R}^m\) of class \(C^1\) and \(h\in\mathcal{H}\) polyhedral, and let \(\bar\partial_\varepsilon h\) be the enlargement \eqref{eq:delpoly}.
Assume \(\bar x\in\mathcal{R}_{h,F}\) and \eqref{eq:Meps-trans}.
Write \(\mathcal{M}\) and \(\mathcal{M}_\varepsilon(x)\) as in Theorem~\ref{thm:Meps}.
Then there exist \(r>0\) and \(\varepsilon_0>0\) such that, for \(x\in\mathcal{M}\cap\overline{B}(\bar x,r)\) and \(0\le\varepsilon\le\varepsilon_0\),
\begin{enumerate}[label=\textnormal{(\roman*)}]
\item\label{it:Meps-freeze} \(I_\varepsilon\bigl(F(x)\bigr)=I\bigl(F(\bar x)\bigr)\);
\item\label{it:Meps-U} \(\mathcal{U}_\varepsilon h(F(x))=\mathcal{U}h(F(\bar x))\);
\item\label{it:Meps-leaf} \(\mathcal{M}_\varepsilon(x)\cap\overline{B}(\bar x,r)=\mathcal{M}\cap\overline{B}(\bar x,r)\);
\item\label{it:Meps-tan} \(T_{\mathcal{M}}(x)=\mathcal{U}_\varepsilon f(x)=\mathcal{U}f(x)\).
\end{enumerate}
In particular the dimension identity assumed in Theorem~\ref{thm:Meps}\ref{it:Meps-dim} holds along \(\mathcal{M}\times[0,\varepsilon_0]\), not necessarily on a full neighbourhood of \((\bar x,0)\).
\end{proposition}
\begin{proof}
Write \(z:=F(\bar x)\) and \(U:=\mathcal{U}h(z)\).

\ref{it:Meps-freeze}. By Proposition~\ref{prop:polyHUeqM} the partial-smoothness manifold of \(h\) at \(z\) coincides locally with \(z+U\). Proposition~\ref{prop:xinU} gives \(z\in\mathcal{U}h(z)\), so \(z+U=\mathcal{U}h(z)\), the ball being centred at \(z\).
Thus if \(x\in\mathcal{M}\) is close to \(\bar x\), then \(F(x)\in z+U\) and \(I(F(x))=I(z)\).
The same identification gives \(J(F(x))=J(z)\): every generator \(c^j\) with \(j\in J(z)\) lies in \(\mathcal{V}h(z)\), hence is orthogonal to \(U\), so the equalities \(\mang{c^j,z}=0\) persist along \(z+U\), while inactive constraints remain inactive for small \(u\in U\).
The inactive inequalities of the representation \eqref{eq:reprPolyh} have a strictly positive slack at \(z\).
Shrink \(r>0\) first so that \(F\bigl(\mathcal{M}\cap\overline{B}(\bar x,r)\bigr)\) lies in a prescribed neighbourhood of \(z\) (possible because \(F\) is continuous at \(\bar x\in\mathcal{M}\)). On that compact piece the slack is then bounded below by some \(\gamma>0\).
Hence \(I_\varepsilon(F(x))\subset I(z)\) whenever \(\varepsilon<\gamma\).
The reverse inclusion \(I(z)=I(F(x))\subset I_\varepsilon(F(x))\) is immediate, so \(I_\varepsilon(F(x))=I(z)\).

\ref{it:Meps-U}. The formulae \eqref{eq:delpoly} and \eqref{eq:polyex} then give \(\bar\partial_\varepsilon h(F(x))=\partial h(z)\), therefore \(\mathcal{U}_\varepsilon h(F(x))=U=\mathcal{U}h(z)\).

\ref{it:Meps-leaf}. Since \(F(x)\in z+U\) one has \(F(x)+U=z+U\), and the implicit-function theorem gives the local identity \(\mathcal{M}_\varepsilon(x)\cap\overline{B}(\bar x,r)=\mathcal{M}\cap\overline{B}(\bar x,r)\).

\ref{it:Meps-tan}. The tangent identity is \eqref{eq:uvefcom} with \(\mathcal{U}_\varepsilon h(F(x))=U=\mathcal{U}h(z)\), together with Theorem~\ref{thm:Meps}'s transversality which makes \(\mathcal{M}\) a \(C^1\) manifold of dimension \(n-\dim V\) near \(\bar x\).
\end{proof}
The dimension identity of Theorem~\ref{thm:Meps}\ref{it:Meps-dim} is an extra hypothesis. Proposition~\ref{prop:Meps-poly} shows that it holds along the leaf for polyhedral \(h\) and small \(\varepsilon\); it is not asserted in a full neighbourhood of \((\bar x,0)\).
The unbiased layer and the max-pooling of Example~\ref{eg:layer} are both polyhedral, so the same freeze applies to them.
The next example computes the freeze for the unbiased layer.
\begin{example}[Example~\ref{eg:layer} continued]\label{eg:layer-freeze}
Let \(h\) and \(F\) be the unbiased layer of Example~\ref{eg:layer}, and take \(\bar x=0\).
The leaf is \(\mathcal{M}=\ker A\), as in Example~\ref{eg:layer}. For \(x\in\mathcal{M}\), \(Ax=0\), so \(F(x)=0\). Each coordinate of \(h\) has two generators, the linear branch and the constant zero, and both are active at the origin, so there is no inactive index. Hence \(I_\varepsilon(F(x))=I(0)\) for every \(\varepsilon\ge 0\).
The activity is frozen along the whole leaf, which is the conclusion of Proposition~\ref{prop:Meps-poly}.
\end{example}
\section{Concluding Remarks}
The exact \(\mathcal{VU}\) algebra of \(f=h\circ F\) is taken from~\cite{Liu2026vu}. The contribution is the passage from an enlargement \(\bar\partial_\varepsilon h\subset\partial_\varepsilon h\) to the composite \(\varepsilon\)-spaces and to approximate leaves.
Theorem~\ref{thm:eps-transfer} transfers the one-sided limits through \(F\) under the transversality \eqref{eq:Meps-trans}; membership in \(\mathcal{R}_{h,F}\) alone gives only the exact chain rule.
Theorem~\ref{thm:Meps} gives inner semicontinuity of the approximate leaf under the nesting \(\mathcal{U}h(F(\bar x))\subset\mathcal{U}_\varepsilon h(F(x))\), and Painlev\'e--Kuratowski convergence of the truncated leaves once the dimension is frozen.
Proposition~\ref{prop:Meps-poly} shows that a polyhedral outer function freezes its activity along the exact leaf for small \(\varepsilon\), so the dimension identity holds on that leaf and not necessarily in a full neighbourhood of \((\bar x,0)\).
Section~5 records that \(h\in\mathcal{H}\) is partly smooth relative to the lineality space, extending~\cite{hare2006functions} from the origin; that section is not used by the two theorems above.
Tilt stability, Hausdorff continuity of \(\mathcal{U}_\varepsilon\) off the leaf, and algorithms are outside the scope. Whether a nonpolyhedral \(h\in\mathcal{H}\) admits a dimension freeze on a set thinner than a neighbourhood of \((\bar x,0)\) is left open.

\section*{Declarations}
\textbf{Corresponding author.}
Shuai Liu, \texttt{shuai0liu@gmail.com}. The author list is complete and the author approved the submission.

\textbf{Author contribution.}
The author is solely responsible for the conception, proofs, and writing.

\textbf{Funding.}
No funding was received for this work.

\textbf{Competing interests.}
The author declares no competing interests.

\textbf{Data and code.}
No data or code were generated. The arguments are mathematical.

\textbf{Ethics.}
Not applicable.

\textbf{Use of AI tools.}
An AI assistant was used to draft and revise expository text, including the abstract, the introduction, the concluding remarks, and transitional sentences, and to check the exposition of proofs. The mathematical statements, the choice of hypotheses, and the approval of every change are the author's.

\textbf{Related work.}
The exact \(\mathcal{VU}\) calculus used as a black box is~\cite{Liu2026vu}, available as arXiv:2609.18683. A preprint will be deposited on arXiv; the manuscript is not under review at any other journal.

\bibliographystyle{amsplain}
\bibliography{VUsublinearComposite}
\end{document}